\documentclass[11pt,letterpaper]{article}
\usepackage[latin9]{inputenc}
\usepackage[letterpaper]{geometry}
\usepackage{amsthm}
\usepackage{amsmath}
\usepackage{amssymb}
\usepackage{esint}

\numberwithin{equation}{section}
\theoremstyle{plain}
\newtheorem{thm}{\protect\theoremname}[section]
  \theoremstyle{plain}
  \newtheorem{lem}[thm]{\protect\lemmaname}
  \theoremstyle{plain}
  \newtheorem{cor}[thm]{\protect\corollaryname}
  \theoremstyle{remark}
  \newtheorem{rem}[thm]{\protect\remarkname}
  \theoremstyle{plain}
  \newtheorem{prop}[thm]{\protect\propositionname}
  \theoremstyle{definition}
  \newtheorem{defn}[thm]{\protect\definitionname}

  \providecommand{\corollaryname}{Corollary}
  \providecommand{\definitionname}{Definition}
  \providecommand{\lemmaname}{Lemma}
  \providecommand{\propositionname}{Proposition}
  \providecommand{\remarkname}{Remark}
\providecommand{\theoremname}{Theorem}

\begin{document}

\title{Complex time evolution in geometric quantization and generalized coherent state transforms}

\author{William D. Kirwin%
\thanks{Mathematics Institute, University of Cologne, Weyertal 86 - 90, 50931
Cologne, Germany.\protect \\
email: will.kirwin@gmail.com%
}, Jos\'e M. Mour\~ao and Jo\~ao P. Nunes%
\thanks{Center for Mathematical Analysis, Geometry and Dynamical Systems and
the Department of Mathematics, Instituto Superior T\'ecnico, Av. Rovisco
Pais, 1049-001 Lisbon, Portugal.\protect \\
email: jmourao@math.ist.utl.pt, jpnunes@math.ist.utl.pt%
}}
\date{}
\maketitle

\begin{abstract}
For the cotangent bundle $T^{*}K$ of a compact Lie group $K$, we study the complex-time evolution of the vertical tangent bundle and the associated geometric quantization Hilbert space $L^{2}(K)$ under an infinite-dimensional family of Hamiltonian flows. For each such flow, we construct a generalized coherent state transform (CST), which is a unitary isomorphism between $L^{2}(K)$ and a certain weighted $L^{2}$-space of holomorphic functions. For a particular set of choices, we show that this isomorphism is naturally decomposed as a product of a Heisenberg-type evolution (for complex time $-\tau$) within $L^{2}(K)$, followed by a polarization--changing geometric quantization evolution (for complex time $+\tau$). In this case, our construction yields the usual generalized Segal--Bargmann transform of Hall. We show that the infinite-dimensional family of Hamiltonian flows can also be understood in terms of Thiemann's ``complexifier'' method (which generalizes the construction of adapted complex structures). We will also investigate some properties of the generalized CSTs, and discuss how their existence can be understood in terms of Mackey's generalization of the Stone-von Neumann theorem.
\end{abstract}

\section{Introduction}

In \cite{Hall02}, Hall initiated the study of the relationship between the coherent-state transform (CST, also known as the generalized Segal--Bargmann transform for compact Lie groups) and geometric quantization. Recall that the geometric quantization (Hilbert space) of a symplectic manifold is the subspace of sections of a certain line bundle on the manifold, called the prequantum line bundle, which are covariantly constant along some choice of polarization\footnote{In geometric quantization, a \emph{polarization} is a complex involutive Lagrangian distribution.}. For example, if the manifold is K\"ahler, one can take the $(1,0)$-tangent bundle for the polarization, and the geometric quantization is the space of square-integrable holomorphic sections of the prequantum line bundle. If the manifold is a cotangent bundle, one can take the complexified vertical tangent bundle for the polarization, and so long as half-forms are included, the geometric quantization is the space of square-integrable functions on the base manifold. The complexification of a compact Lie group $K$ is diffeomorphic to the cotangent bundle of $K$, and as such it has both of these structures. Hall showed that the CST can be understood in terms of geometric quantization as a unitary isomorphism between the vertically polarized Hilbert space $L^2(K)$ and the Hilbert space for the standard K\"ahler polarization.

The problem of choice of polarization is a fundamental issue in geometric quantization. In good cases, one might hope that the quantization is independent of this choice, but it turns out that such a hope is simply too optimistic. For example, each almost complex structure on a symplectic manifold gives rise to an almost-K\"ahler quantum Hilbert space, and one can attempt to compare these Hilbert spaces by forming a Hilbert bundle over the neighborhood of a point in the space of almost complex structures. This bundle has a natural connection (generalizing the connections of Axelrod--Della Pietra--Witten\cite{Axelrod-DellaPietra-Witten} and Hitchin \cite{Hitchin90}) which is given by projecting the trivial connection in the trivial bundle whose fiber is the space of all square-integrable sections of the prequantum line bundle onto the almost-holomorphic subbundle. In the case that the symplectic manifold is a symplectic vector space and one restricts to translation invariant complex structures, this connection is known to be projectively flat \cite{Axelrod-DellaPietra-Witten,Kirwin-Wu06}. On the other hand, if one considers the full family of almost complex structures, Foth and Uribe have shown that the connection is never projectively flat, even semiclassically \cite{Foth-Uribe}. The lesson here is that one can expect projective flatness only for certain restricted families of complex structures.

In general, to obtain the correct quantization one must include half-forms. The Axelrod--Della Pietra--Witten/Hitchin connection generalizes naturally to a connection induced by the pairing defined by the half-forms, which is known as the BKS (Blattner--Kostant--Sternberg) pairing. However, the pairing itself is in general not equal to the parallel transport of the connection it induces and is (in general) not unitary. Examples of nonunitary BKS pairing maps are given by torus-invariant K\"ahler structures on compact symplectic toric manifolds \cite{Baier-Florentino-Mourao-Nunes,kirwin-mourao-nunes10}).

For $T^*K$, Hall shows in \cite{Hall02} that the CST is equal to the BKS pairing map between the vertically polarized Hilbert space and the Hilbert space corresponding to the  K\"ahler polarization, induced by the standard diffeomorphism from $T^*K$ to $K_{\mathbb C}$. A  one-real-parameter family of K\"ahler polarizations, ${\cal P}^{it}, t >0$ (see (\ref{tauvectors}) below), containing the standard  K\"ahler polarization, for $t=1$,  and degenerating to the vertical polarization for $t=0$, was studied in \cite{Florentino-Matias-Mourao-Nunes05,Florentino-Matias-Mourao-Nunes06}, where it has been shown that the  BKS pairing map, between the Hilbert spaces corresponding to $t_1>0$ and $t_2>0$, is equal to the parallel transport of a heat equation type connection similar to the ones considered in \cite{Axelrod-DellaPietra-Witten, Hitchin90, Kirwin-Wu06}. By adding the fiber $L^2(K) \rightarrow \{t=0\}$ to the corresponding Hilbert bundle, the connection extends and then the results of \cite{Hall02} imply that the CST is equal to the parallel transport of the extended connection from $t_1=0$ to $t_2>0$.

In this article, we will introduce an infinite dimensional family of $K\times K$-invariant K\"ahler structures on $T^{*}K$ which includes the one-parameter family of K\"ahler complex structures mentioned above, and we will show that there exist generalized CSTs (which are unitary isomorphisms) between the corresponding K\"ahler quantizations and $L^{2}(K)$, and which intertwine natural actions of $K\times K$.

Although we will first construct the K\"ahler structures by hand, we will see that they can be understood in terms of the ``complexifier'' approach due to Thiemann \cite{Thiemann96}, which in turn is related to adapted complex structures. Let $(M,g)$ be a real-analytic Riemannian manifold, and let $\kappa$ be the norm-squared function on the fibers of $T^{*}M$. It was shown independently but essentially coincidentally by Lempert--Sz\H{o}ke \cite{Lempert-Szoke-91,Szoke-91} and Guillemin--Stenzel \cite{Guillemin-Stenzel-91,Guillemin-Stenzel-92} that there exists a tubular neighborhood of $M$ in $T^{*}M$ which admits a (unique) K\"ahler structure whose K\"ahler $2$-form is the standard cotangent bundle symplectic form and such that the restriction of the K\"ahler metric to $M$ is $g$, multiplication by $-1$ in the fibers is an antiholomorphic involution, and  $\kappa$ is a  K\"ahler potential. It turns out that the standard complex  structure on $K_{\mathbb{C}}$ is the adapted complex structure (which exists globally in this case) associated to the metric induced by a choice of bi-invariant inner product on the Lie  algebra of $K$.

The adapted complex structure can be constructed by analytically continuing the Hamiltonian flow $\Phi_{t}^{\kappa/2}:T^{*}M\mapsto T^{*}M$ of $\kappa/2$, which is just the geodesic flow, to  time $t=i \,  (=\sqrt{-1})$ \cite{Hall-Kirwin}. There are several ways to understand the imaginary-time geodesic flow; for example, the pushforward of the vertical tangent bundle by the time-$t$ geodesic flow yields a family of Lagrangian distributions which can be analytically continued to imaginary time, and when we set $t=i$, we obtain the $(1,0)$-tangent bundle of the adapted complex structure. At the level of functions, it can be shown that holomorphic functions (with respect to the adapted complex structure) are of the form $f\circ\pi\circ\Phi_{i}$, where $\Phi_{t}:T^{*}M\rightarrow T^{*}M$ is the geodesic flow, $f$ is a function on $M \subset T^*M$, admitting analytic continuation, and $\pi:T^{*}M\rightarrow M$ is the natural projection. The function $f\circ\pi\circ\Phi_{i}^{\kappa/2}$ at the point $(x,p)\in T^{*}M$ is to be understood as the analytic continuation of the real family $t\mapsto f\circ\pi\circ\Phi_{t}^{\kappa/2}(x,p)$ to $t=i$. 

One can generalize the adapted complex structure slightly by simply evaluating at time $\tau\in\mathbb{C}$. The analytic continuation will exist for $\tau$ in some neighborhood of the origin (including $i$ if one is considering the tube on which the adapted complex structure exists), and the resulting complex structure is positive, and hence K\"ahler, if and only if $\mathrm{\tau \in \mathbb{C}^{+}},$ i.e. ${\rm Im}(\tau) >0$. Lempert and Sz\H{o}ke have recently studied the $\tau$-(in)dependence of K\"ahler quantization with respect to the time-$\tau,\ \tau\in\mathbb{C}^{+}$, adapted complex structures \cite{Lempert-Szoke10}.

The method Thiemann proposes is a generalization of this ``time-$\tau$'' geodesic flow as follows. Let $h$ be a choice of ``complexifier'' function on $T^{*}M$, and denote the time-$t$ Hamiltonian flow of $h$ by $\Phi_{t}^{h}.$ Thiemann proposes to define a complex structure by declaring that the holomorphic functions are $f\circ\pi\circ\Phi_{\tau}^{h}$, or, equivalently, that the $(1,0)$-tangent bundle is given by the analytic continuation of the time-$t$ pushforward of the vertical tangent bundle by $\Phi_{t}^{h}$ evaluated at $t=\tau$. Of course, convergence issues abound, and there is no reason to expect that for an arbitrary function $h$ such analytic continuations exist. On the other hand, the adapted complex structure shows that if $h=\kappa/2$ is half the norm-squared function of some choice of metric on $M$, then for $\tau=i$ Thiemann's method produces the adapted complex structure on a tubular neighborhood of $M$ in $T^{*}M$.

In \cite{Hall-Kirwin2}, Hall and the first author have shown that for $h=\kappa/2$, Thiemann's method can be generalized by replacing the the canonical cotangent symplectic form $\omega^{T^{*}K}$ by a $2$-form $\omega^{T^{*}K}+\pi^{*}\beta$ which is the canonical $2$-form twisted by a magnetic field on $M$. Again, the time-$\tau$ Hamiltonian flow of $\kappa$ yields K\"ahler structure on a tubular neighborhood of $M$ for $\tau$ in some neighborhood of the origin in $\mathbb{C}^{+}$ (and, clearly, if $\beta=0$, one obtains the usual adapted complex structure). This construction is locally equivalent to using the standard symplectic form and replacing the K\"ahler potential $\kappa$ by half the norm-squared of the ``canonical'' momentum which depends on a choice of local magnetic potential. Hence, these magnetic complex structures are more examples where Thiemann's complexifier method is successful.

As mentioned above, although we will first construct them by hand, we will also show that the infinite dimensional family of K\"ahler structures that we construct on $T^{*}K$ can be understood in terms of Thiemann complex structures for a certain class of complexifier functions $h$, thus providing many more example where Thiemann's method provides convergent results (indeed, the complex structures we construct exist on all of $T^{*}K$, not just a tubular neighborhood of $K$). We will also show that the time-$\tau$ flows which yield our K\"ahler structures can be lifted to the prequantum line bundle, and hence yields maps from the cotangent bundle (i.e. vertical polarization) quantization of $T^{*}K$ to the K\"ahler quantizations associated to our family. In some sense, these maps are the first ``half'' of our generalized CSTs. The second ``half'' arises when we force the maps to be unitary isomorphisms.

More precisely, let $e^{-i\tau\hat{h}}$ denote the action, defined by geometric quantization, of the flow of a complexifier $h$ on sections of the prequantum bundle. It will turn out that if $\mathcal{H}_{0}$ is the cotangent bundle quantization of $T^{*}K$ and $\mathcal{H}_{\tau}$ is the K\"ahler quantization of $T^{*}K$ with respect to the time-$\tau$ flow of the vertical polarization, then
\[
e^{-i\tau\hat{h}}:\mathcal{H}_{0}\rightarrow\mathcal{H}_{\tau}
\]
is a densely defined linear operator which is only unitary for $\tau$ real. Thus, we will look for an endomorphism $E(\tau,h):\mathcal{H}_{0}\rightarrow\mathcal{H}_{0}$ such that the composition
\[
e^{-i\tau\hat{h}}\circ E(\tau,h):\mathcal{H}_{0}\rightarrow\mathcal{H}_{\tau}
\]
intertwines the geometric quantization quantization action of $K\times K$ and is unitary. We will show that for the class of complexifiers that we consider, such endomorphisms $E$ exist.  The usual CST introduced by Hall depends on a real parameter $t$. In terms of Thiemann's complexifier method, the parameter-$t$ CST $U_{it}$ arises from the time-$it$ geodesic flow. In fact, when $h=\kappa/2$ and $\tau=it$, one can take
\[
E(it,\kappa/2)=e^{-t\left(-\frac{1}{2}\Delta+\frac{|\rho|}{2}^{2}\right)},
\]
where $\Delta$ is the (negatively-defined) Laplacian on $K$ and $\rho$ is half the sum of the positive roots, and Hall's CST $U_{it}$ can be written
\[
U_{it}=e^{t\hat{\kappa}/2}\circ e^{-t\left(-\frac{1}{2}\Delta+\frac{|\rho|}{2}^{2}\right)}.
\]
By defining the vertical polarization quantization of $h= \kappa/2$ by\footnote{We will comment on the value of the additive constant in the right-hand side of (\ref{hsch}) in Section \ref{gcsts}.}
\begin{equation}
\label{hsch}
Q(h) = - \frac 12 \Delta + \frac{|\rho|^2}2,
\end{equation}
we see that $U_{it}$ has the form
\[
U_{it} = e^{-i \, (it) \hat h} \circ  e^{-i \, (-it) Q(h)} .
\]
The reason behind this form of the CST, as a composition of $-it$-Heisenberg evolution (with fixed polarization) followed by a $+it$ (polarization changing) geometric quantization time evolution, is the Mackey's generalization of the Stone-von Neumann theorem, as we will discuss briefly in  section \ref{gcsts} and in more detail in  \cite{Kirwin-Mourao-Nunes}.

In this case, the family of transforms $E(it,k/2)$ satisfies the semigroup property
\[
E(it_1,k/2)  \circ  E(it_2,k/2) =  E(i(t_1+t_2),k/2)  .
\]
We will show that the generalized $h$-CSTs do not have this property in general, although we will examine the phenomenon in more detail in future work.

\section{Preliminaries}

Let $K$ be a Lie group of compact type. Denote the Lie algebra of $K$ by $\mathfrak{k}:=T_{1}K$. Let $B$ be a positive definite bi-invariant bilinear form on $\mathfrak{k}$ (e.g., one can take $B$ to be a multiple of the negative of the Killing form). We will choose a specific normalization for $B$ in the second paragraph below. We will permanently identify $\mathfrak{k}\simeq\mathfrak{k}^{*}$ using the map $B:\mathfrak{k}\rightarrow\mathfrak{k}^{*}$ given by $X\mapsto B(X,\cdot).$ Using left translation, we will identify $TK\simeq K\times\mathfrak{k}$ and $T^{\ast}K\simeq K\times\mathfrak{k}^{\ast}\simeq K\times\mathfrak{k}$. Throughout, since $T_{(x,Y)}(T^{\ast}K)\simeq\mathfrak{k}\oplus\mathfrak{k}$, we will write vectors on $T^{\ast}K$ as block column vectors $\begin{pmatrix}X\\Y\end{pmatrix}$, where $X,Y\in\mathfrak{k}$. The canonical $1$-form on $T^{\ast}K$ is
\[
\theta_{(x,Y)}\left(\begin{pmatrix}X\\
V
\end{pmatrix}\right)=B(Y,X).
\]
One may then compute that the canonical symplectic form $\omega^{T^{\ast}K}=-d\theta$ on $T^{\ast}K$ is
\[
\omega_{(x,Y)}^{T^{\ast}K}\left(\begin{pmatrix}X\\
V
\end{pmatrix},\begin{pmatrix}Z\\
W
\end{pmatrix}\right)=B(W,X)-B(V,Z)+B(Y,[X,Z]).
\]
It will be occasionally useful to note that $\omega^{T^{\ast}K}$ can be written in block form as
\begin{equation}
\omega_{(x,Y)}^{T^{\ast}K}=\begin{pmatrix}-ad_{Y} & \mathbf{1}\\
-\mathbf{1} & 0
\end{pmatrix},\label{eqn:symp-block}
\end{equation}
with the convention that the product of a row vector and a column vector is the inner product $B$.

With $n:=\dim K$, let $\left\{ T_{j}\right\} _{j=1,...,n}$ be a $B$-orthonormal basis of the Lie algebra $\mathfrak{k}$ and let $\{X_{j}\}_{j=1,\dots,n}$ be the basis of left-invariant vector fields on $K$ which is equal to $\left\{ T_{j}\right\} $ at the identity. Let $\{y^{j}\}_{j=1,\dots,n}$ be the coordinates on $\mathfrak{k}$ with respect to $\left\{ T_{j}\right\} $, and let $\{w^{j}\}_{j=1,\dots,n}$ be the basis of left-invariant 1-forms on $K$ dual to the vector fields $X_{j}$. We will also denote the pullback to $T^{*}K$ of $w^{j}$ via the canonical projection by $w^{j}$. Let $\left\{ \tilde{X}_{j}\right\}$ be the basis of right-invariant vector fields on $K$ which is equal to $\left\{ T_{j}\right\}$ at the identity, and let $\left\{ \tilde{y}^{j}\right\}$ and $\left\{ \tilde{w}^{j}\right\}$ be the associated coordinates on $\mathfrak{k}$ and right-invariant forms on $K$.

The bilinear form $B$ induces a metric on $K$ whose volume form is a (bi-invariant) Haar measure $dx$ which can be expressed as $dx=w^{1}\wedge\cdots\wedge w^{n}.$ We normalize $B$ so that the resulting volume $\mathrm{vol}K:=\int_{K}dx$ is equal to $1$.

Since $\Theta = \sum_{j=1}^n y^j w^j$, we obtain
\begin{equation}
\omega^{T^{*}K}=\sum_{j=1}^{n}\left(w^{j}\wedge dy^{j}+\frac{1}{2}\sum_{k,l=1}^{n}C_{kl}^{j} \, y^j w^{k}\wedge w^{l}\right),\label{eq:sympform}
\end{equation}
where $C_{kl}^{j}$ denote the (totally antisymmetric) structure constants of $\mathfrak{k}$ in the basis $\{X_j\}_{j=1, \dots ,n}$. We see that the Liouville measure on $T^{*}K$ can be expressed as
\[
dx\, dY,
\]
where $dY=dy^{1}\wedge\cdots\wedge dy^{n}$ is the Lebesgue measure on $\mathfrak{k}$.

Since $K$ is of compact type, it admits a unique complexification $K_{\mathbb{C}}$. Let $\hat{K}$ denote the set of equivalence classes of irreducible representations of $K$ and recall that they are all finite dimensional and unitary. There is a $1$-to-$1$ correspondence between irreducible representations of $K$ and finite-dimensional (non-unitary) irreducible representations of $K_{\mathbb{C}}$. Recall that if $\rho$ is a finite-dimensional irreducible representation of $K_{\mathbb{C}}$, then its restriction to $K$ is the corresponding element of $\hat{K}$. Hence, we denote the set of equivalence classes of irreducible finite-dimensional representations of $K_{\mathbb{C}}$ also by $\hat{K}$. We will not explicitly distinguish between a representation of $K_{\mathbb{C}}$ and its restriction to $K$, although it will be clear from the context.

The diffeomorphism
\begin{eqnarray*}
T^{*}K\cong K\times\mathfrak{k} & \to & K_{\mathbb{C}}\\
(x,Y) & \mapsto & xe^{iY}
\end{eqnarray*}
can be used to pull back the canonical complex structure from $K_{\mathbb{C}}$ to $T^{*}K$, so that $(T^{*}K,\omega)$ becomes a K\"ahler manifold. We refer to this K\"ahler (complex) structure on $T^{*}K$ as the standard K\"ahler (complex) structure on $T^{*}K$.

\section{Geometric quantization of $T^{*}K$}

Let $L\rightarrow T^{*}K$ be a hermitian line bundle with hermitian structure $h^{L}$ and compatible connection $\nabla^{L}$ with curvature $-i\omega^{T^{*}K}$. The line bundle $L$ is called a prequantum bundle for $T^{*}K$. Since $\omega^{T^{*}K}$ is exact, $L$ is globally trivializable, and each choice of symplectic potential induces a trivialization. We will use the canonical $1$-form $\theta$, so that $\nabla^{L}=d+i\Theta$. The geometric prequantization of $(T^{*}K,\omega^{T^{*}K})$ is the space of sections of a prequantum line bundle $L\rightarrow T^{*}K$ which are square integrable with respect to the inner product
\[
\left\langle s,t\right\rangle :=\int_{T^{*}K}h^{L}(s(x,Y),t(x,Y))\, dxdY.
\]

There are two standard ways to proceed from the prequantization of $T^{*}K$ to a quantization of $T^{*}K$. The first is to use the cotangent bundle structure, while the second relies on the diffeomorphism $K_{\mathbb{C}}\simeq T^{*}K$. In the first case, half-forms must be included in order to obtain a nonzero quantum Hilbert space. In the second case, the need for half-forms is less evident. Nonetheless, it is commonly believed that the half-form correction is necessary also in the K\"ahler case and many arguments have been presented in its favor: the half-form correction renders the BKS pairing map unitary in the quantization of vector spaces with translation invariant complex structures \cite{Axelrod-DellaPietra-Witten,Kirwin-Wu06} and of Abelian varieties \cite{Baier-Mourao-Nunes} and allows for a transparent explanation of the vacuum energy shift in the K\"ahler quantization of symplectic toric varieties with toric K\"ahler structures \cite{kirwin-mourao-nunes10}.

A \emph{polarization $\mathcal{P}$} of a symplectic manifold is a complex involutive Lagrangian distribution. The dual $\mathcal{P}^{*}$ is the subbundle of the complexified cotangent bundle consisting of $1$-forms which vanish when restricted to $\overline{\mathcal{P}}$. The canonical bundle $\mathcal{K}^{\mathcal{P}}$of a polarization $\mathcal{P}$ is the top exterior power of $\mathcal{P}^{*}$. The $(1,0)$-tangent bundle of a K\"ahler manifold is a polarization whose canonical bundle is the usual canonical bundle. On the other hand, the complexified vertical tangent bundle on $T^{*}K$ is a polarization, which we will denote by $\mathcal{P}^{0}$, whose canonical bundle $\mathcal{K}^{0}$ is the subbundle of $\bigwedge^{n}(T^{*}K)$ whose sections are $n$-forms which evaluate to zero on $\partial/\partial y^{j},\ j=1,...,n$. Hence,
\[
\Gamma(\mathcal{K}^{0})=C^{\infty}(T^{*}K)\otimes dx.
\]

Let $\mathcal{P}$ be a polarization on $T^{*}K$. Suppose that $\mathcal{K}^{\mathcal{P}}$ admits a square root $\sqrt{\mathcal{K}^{\mathcal{P}}}$ and fix a choice of square root. If $\mathcal{K}^{\mathcal{P}}$ is trivial with global nowhere vanishing section $\Omega$, then $\sqrt{\mathcal{K}^{\mathcal{P}}}$ can also be chosen to be trivial with a trivializing section which squares to $\Omega$ and which we therefore denote by $\sqrt{\Omega}$. In this case, $\Gamma(\sqrt{\mathcal{K}^{\mathcal{P}}})=C^{\infty}(T^*K)\otimes\sqrt{\Omega}$.

The bundle $\sqrt{\mathcal{K}^{\mathcal{P}}}$ is called a half-form bundle, and it comes equipped with a canonical hermitian structure known as the half-form pairing. If $\mathcal{P}$ is the $(1,0)$-tangent bundle of a K\"ahler complex structure on $T^{*}K$, then the half-form pairing is given by comparison to the Liouville form; explicitly, for $\mu,\mu^{\prime}\in\sqrt{\mathcal{K}_{(x,Y)}^{\mathcal{P}}}$, the pairing $(\mu,\mu^{\prime})$ is the unique complex number determined by
\[
\left(\frac{1}{2i}\right)^{n}\bar{\mu}^{2}\wedge\left(\mu^{\prime}\right)^{2}=(\mu,\mu^{\prime})^2 \frac{\left(\omega^{T^{*}K}\right)^{n}}{(-1)^{n(n-1)/2}n!},
\]
where the branch of the square root is chosen  so that $(\mu,\mu)>0$ \cite{Woodhouse}. The constants are chosen so that on $(\mathbb{R}^{2},dx\wedge dy)$ equipped with the standard complex structure $z=x+iy$, one has $(\sqrt{dz},\sqrt{dz})=1$.

The \emph{half-form corrected} \footnote{The half-form correction is also known as the\emph{ metaplectic correction}.} \emph{geometric quantization} $\mathcal{H}_{\mathcal{P}}$ of $(T^{*}K,\omega^{T^{*}K},\mathcal{P})$ is the space of sections of $L\otimes\sqrt{\mathcal{K}^{\mathcal{P}}}$ which are covariantly constant along $\overline{\mathcal{P}}$ and square-integrable with respect to the inner product induced by $h^{L}$ and the half-form pairing on $\sqrt{\mathcal{K}^{\mathcal{P}}}$.

When $\mathcal{P}$ is the $(1,0)$-tangent bundle of the standard complex structure on $T^{*}K\simeq K_{\mathbb{C}}$, the canonical bundle, and hence the half-form bundle, are trivializable and Hall has shown that the quantum Hilbert space, which we will denote by $\mathcal{H}_{i}$ for reasons which will become clear in Section \ref{sec:Time-evolution}, is given by
\begin{multline*}
\mathcal{H}_{i}=\Big\{ F(xe^{iY})e^{-\frac{|Y|^{2}}{2}}\otimes\sqrt{\Omega}:F\text{ is holomorphic,} \\ \text{and}\int_{K\times\mathfrak{k}}\left|F\right|^{2}e^{-\left|Y\right|^{2}}(\sqrt{\Omega},\sqrt{\Omega})dxdY<\infty\Big\}
\end{multline*}
\cite{Hall02} where $\Omega$ is a certain holomorphic $(n,0)$-form on $T^{*}K$ (the wedge product of the $\Omega_{\tau=i}^{j}$'s given in Lemma \ref{lem:frames}).

When $\mathcal{P}$ is the complexification of the vertical tangent bundle of $T^{*}K$, a section which is covariantly constant along $\mathcal{P}$ is determined by its value on the zero section, whence the half-form corrected quantization $\mathcal{H}_{0}$ (again, the reason for the choice of notation will become clear in Section \ref{sec:Time-evolution}) is naturally isomorphic to
\[
\mathcal{H}_{0}=L^{2}(K,dx).
\]

Let $\mathcal{C}$ denote analytic continuation from $K$ to $K_{\mathbb{C}}$ and $\Delta$ the (negatively defined) Laplace operator on $K$. Recall the \textit{coherent state transform} (CST) of Hall \cite{Hall94} 
\begin{eqnarray*}
C_{t}:L^{2}(K,dx) & \to & L_{\mbox{hol}}^{2}(K_{\mathbb{C}},d\nu_{t})\\
f & \mapsto & C_{t}(f)=\mathcal{C}\circ e^{\frac{t}{2}\Delta}f,
\end{eqnarray*}
where $t>0$ and $L_{\mbox{hol}}^{2}(K_{\mathbb{C}},d\nu_{t})$ denotes the space of holomorphic functions on $K_{\mathbb{C}}$ which are $L^{2}$ with respect to the so-called averaged heat kernel measure $d\nu_{t}$ which is proportional to $e^{-\frac{\left|Y\right|^{2}}{t}} \eta(Y) dxdY$ (see \cite{Hall94,Hall02} or \ref{eq:dmu}). Here, $\eta$ is the $Ad$-invariant function on $\mathfrak{k}$ which is determined on a chosen Cartan subalgebra of $\mathfrak{k}_{\mathbb{C}}$ by \cite{Hall97}
\[
\eta(Y)=\prod_{\alpha\in\Delta^{+}}\frac{\mathrm{sinh}\,\alpha(Y)}{\alpha(Y)},
\]
where $\Delta^{+}$ is the associated set of positive roots.

Hall proves:

\begin{thm}
For all $t>0$, $C_{t}$ is a unitary isomorphism of Hilbert spaces.
\end{thm}

In \cite{Hall02,Florentino-Matias-Mourao-Nunes05,Florentino-Matias-Mourao-Nunes06}, it was shown that the CST can be understood in terms of the geometric quantizations of $T^{*}K$ arising from the vertical and K\"ahler polarizations. Indeed, one has the isomorphism of Hilbert spaces
\begin{eqnarray*}
\mathcal{H}_{i} & \cong & L_{\mbox{hol}}^{2}(K_{\mathbb{C}},d\nu_{1})\\
F(xe^{iY})e^{-\frac{|Y|^{2}}{2}}\otimes\sqrt{\Omega} & \mapsto & F(xe^{iY}).
\end{eqnarray*}
Hence, one can study the BKS pairing map between $L^{2}(K,dx)$ and $L_{\mbox{hol}}^{2}(K_{\mathbb{C}},d\nu_{t}$. Hall computed this BKS pairing, with the following result.

\begin{thm}
The BKS pairing map $\mathcal{H}_{0}\to\mathcal{H}_{i}$ coincides with the CST.
\end{thm}

In \cite{Florentino-Matias-Mourao-Nunes05,Florentino-Matias-Mourao-Nunes06}, the authors considered a one-real-parameter family of polarizations connecting the vertical polarization with the K\"ahler polarization on $T^{*}K$. The corresponding BKS pairing maps were shown to be unitary. Degenerating one of the K\"ahler polarizations to the vertical polarization, one recovers Hall's result relating the BKS pairing with the CST. This family of polarizations gives then a family of quantizations forming a Hilbert space bundle with continuous hermitian structure over $\mathbb{R}_{\geq0}$.

\section{Families of K\"ahler structures on $T^{*}K$}

In this section, we will describe an infinite-dimensional family of K\"ahler structures on $T^{\ast}K$ that are compatible with the canonical symplectic structure. This family contains the one-parameter deformations of the standard K\"ahler structure that are considered in \cite{Florentino-Matias-Mourao-Nunes05,Florentino-Matias-Mourao-Nunes05,Hall-Kirwin,Lempert-Szoke10}. The K\"ahler structures will be constructed using a certain class of functions $h:T^{*}K\rightarrow\mathbb{R}$ which generalize the standard Hamiltonian of a free particle moving on $K$ (i.e. half of the norm of the momentum squared).

In the next section, we will study this family of K\"ahler structures at the level of their $(1,0)$-tangent bundle and in terms of the associated holomorphic functions, and in particular we will see that these K\"ahler structures arise as the ``time-$\tau$'' flow generated by  the associated function $h$ as per the method of Thiemann \cite{Thiemann96}. As discussed in the introduction, we will refer to $h$ as a (Thiemann) \emph{complexifier }function. The standard complex structure on $T^{*}K$ and the one-parameter family of deformations of it which are studied in \cite{Florentino-Matias-Mourao-Nunes05,Florentino-Matias-Mourao-Nunes05,Hall-Kirwin,Lempert-Szoke10} are all associated to the ``kinetic energy'' complexifier $\frac{1}{2}\kappa(x,Y)=\frac{1}{2}B(Y,Y)=\frac{1}{2}\left|Y\right|^{2}.$

\bigskip
We consider the family of complexifier functions $h:T^{*}K\cong K\times\mathfrak{k}\to\mathbb{R}$ such that
\begin{align}
\mbox{1.}\ \  & h(x,Y)\mbox{ is an }Ad\mbox{-invariant smooth function depending only on }Y\in\mathfrak{k},\nonumber \\
\mbox{2.}\ \  & \mbox{the Hessian }H(Y)\mbox{ of }h\mbox{ is positive definite at every point }Y\in\mathfrak{k},\mbox{ and }\label{eq:3props}\\
\mbox{3.}\ \  & \mbox{the operator norm }||H(Y)||\mbox{ has nonzero lower bound.}\nonumber
\end{align}
We henceforth fix a choice of such complexifier $h$.

Let $u(Y)\in\mathfrak{k}$ be the $B$-gradient of the function $h$ at $Y$
\[
u(Y):=B^{-1}(dh).
\]
In local coordinates, $u=\sum_{j=1}^n u^{j}T_{j}$ with $u^{j}=\partial h/\partial y^{j}$. Let $\alpha$ be the induced map of $T^{*}K$ to itself given by
\begin{equation}
\begin{array}{rcl}
\alpha:T^{*}K & \to & T^{*}K\\
(x,Y) & \mapsto & (x,u(Y)).
\end{array}
\end{equation}

\begin{lem}
The map $\alpha$ is a diffeomorphism of $T^{*}K$.
\end{lem}

\begin{proof}
Since the Hessian of $h$ is nondegenerate at every point, the derivative of $\alpha$ is always nonsingular and $\alpha$ is a local diffeomorphism whence the image of $\alpha$ is open and nonempty. On the other hand,
\begin{align*}
u(Y^{\prime})-u(Y)&=\int_{0}^{1}H(Y+t(Y^{\prime}-Y))(Y^{\prime}-Y)dt\\
&=\left(\int_{0}^{1}H(Y+t(Y^{\prime}-Y))dt\right)(Y^{\prime}-Y).
\end{align*}
Since $H$ is nondegenerate and positive definite at every point, $\alpha$ is injective.

To show surjectivity, it is clear that if $u_{n}$ is a sequence of points in the image of $\alpha$ converging to $u$, and such that the sequence of preimages is bounded, then $u_{n}$ converges to some $u$ also in the image of $\alpha$. On the other hand, let $u_{0}=h(0)$ and $0\neq v\in\mathfrak{k}$, so that for $t\geq0$,
\[
B(v,u(tv)-u_{0})=\int_{0}^{t}B(v,H(tv)v)\, dt\geq mt,
\]
for some ($v$ independent) constant $m>0$, since $H$ is always positive definite with operator norm bounded away from zero. Therefore, as $t\to+\infty$, $B(v,u(tv))\to+\infty$ as well. It follows that if we take an unbounded sequence of points $Y_{n}=t_{n}v_{n},n=1,2,\dots$, with $||v_{n}||=1$ and $\{t_{n}\}_{n\in\mathbb{N}}$ unbounded, then the sequence $\{B(v_{n},u_{n})\}_{n\in\mathbb{N}}$ is unbounded, so that the sequence $\{u_{n}\}_{n\in\mathbb{N}}$ is not convergent. Therefore the image of $\alpha$ is also closed and $\alpha$ is surjective.
\end{proof}

\bigskip
For
\[
\tau:=\tau_{1}+i\tau_{2}\in\mathbb{C}^{+}:=\{\tau\in\mathbb{C}:\tau_{2}:=\mathrm{Im}\tau>0\},
\]
consider the diffeomorphism
\begin{equation}
\begin{array}{ccccc}
T^{*}K & \overset{\alpha}{\rightarrow} & T^{*}K & \overset{\psi_{\tau}}{\rightarrow} & K_{\mathbb{C}}\\
(x,Y) & \mapsto & (x,u(Y)) & \mapsto & xe^{\tau u(Y)}.
\end{array}
\end{equation}
The diffeomorphism $\psi_{\tau}$ is studied in \cite{Florentino-Matias-Mourao-Nunes05,Florentino-Matias-Mourao-Nunes06,Hall-Kirwin,Lempert-Szoke10}. Recall that the pullback by $\psi_{\tau}$ of the canonical complex structure on $K_{\mathbb{C}}$ defines, together with the canonical symplectic structure $\omega$ on $T^{*}K$, a K\"ahler structure with symplectic potential $\tau_{2}||u||^{2}$ \cite{Lempert-Szoke10}. Let $J^{\tau}\in\mathrm{End}(T(T^{*}K))$ be the pullback of the canonical complex structure on $K_{\mathbb{C}}$ to $T^{*}K$ by $\alpha\circ\psi_{\tau}$.

\begin{thm}
\label{kahler} For any $\tau\in\mathbb{C}^{+}$, the pair $(\omega,J^{\tau})$ defines a K\"ahler structure on $T^{*}K$, with K\"ahler potential
\[
\kappa(Y)=2\tau_{2}(B(Y,u(Y))-h(Y)).
\]
In particular, the corresponding K\"ahler polarization $\mathcal{P}^{\tau}:=T^{(1,0)}T^{*}K$ is positive.
\end{thm}

To prove the theorem we will need a few auxiliary lemmas, which turn out to be repeatedly useful in what follows.

\begin{lem}
\label{lem:frames}The $(1,0)$-tangent space of the complex structure $J^{\tau}$ at the point $(x,Y)$ is
\begin{equation}
\mathcal{P}_{(x,Y)}^{\tau}=\left\{ \left(\begin{array}{c}
\mathrm{ad}_{u(Y)}^{-1}\left(\mathbf{1}-e^{\bar{\tau}\,\mathrm{ad}_{u(Y)}}\right)X\\
H(Y)^{-1}X
\end{array}\right):X\in\mathfrak{k}\right\} .\label{tauvectors}
\end{equation}
In particular, the $(1,0)$-tangent space at $(x,Y)$ is spanned by the left $K_{\mathbb C}$-invariant holomorphic frame
\[
Z_{j}^{\tau}:=\sum_{k=1}^{n}\left[\frac{e^{i\tau_{2}ad_{u(Y)}}}{2i\mathrm{sin}(\tau_{2}ad_{u(Y)})}\left(\mathbf{1}-e^{\bar{\tau}\,\mathrm{ad}_{u(Y)}}\right)\right]_{j}^{k}X_{k}+\left[\frac{ad_{u(Y)}e^{i\tau_{2}ad_{u(Y)}}}{2i\mathrm{sin}(\tau_{2}ad_{u(Y)})}\cdot H(Y)^{-1}\right]_{j}^{k}\frac{\partial}{\partial y^{k}}.
\]
The $(1,0)$-cotangent space at the point $(x,Y)$ is 
\[
\left\{ \left(e^{\tau ad_{u(Y)}^{*}}w,\frac{e^{\tau ad_{u}^{*}}-1}{ad_{u}^{*}}Hw\right):w\in\mathfrak{k}^{*}\right\} .
\]
In particular, the $\left\{ Z_{j}\right\} $-dual holomorphic frame of type-$(1,0)$ forms is given by $\{\Omega_{\tau}^{i}\}_{i=1,\dots,n}$ where
\begin{equation}
\Omega_{\tau}^{j}=\sum_{k=1}^{n}\left[e^{-\tau ad_{u(Y)}}\right]_{k}^{j}w^{k}+\left[\frac{1-e^{-\tau ad_{u(Y)}}}{ad_{u(Y)}}H(Y)\right]_{k}^{j}dy^{k}.\label{tauforms}
\end{equation}
\end{lem}

\begin{proof}
The derivative of $\psi_{\tau}$ in the case $\tau=i$ has been computed in \cite{Hall97} and it is straightforward to generalize the result to more general $\tau$ to obtain
\[
[(\psi_{\tau})_{*}]_{(x,Y)}=\begin{pmatrix}e^{-\tau_{1}ad_{Y}}\mathrm{cos}(\tau_{2}ad_{Y}) & \frac{1-e^{-\tau_{1}ad_{Y}}\mathrm{cos}(\tau_{2}ad_{Y})}{ad_{Y}}\\
-e^{-\tau_{1}ad_{Y}}\mathrm{sin}(\tau_{2}ad_{Y}) & \frac{e^{-\tau_{1}ad_{Y}}\mathrm{sin}(\tau_{2}ad_{Y})}{ad_{Y}}
\end{pmatrix},
\]
where the blocks are relative to the decompositions $T_{(x,Y)}TK\simeq\mathfrak{k}\oplus\mathfrak{k}$ and the decomposition $T_{g}K_{\mathbb{C}}\simeq\mathfrak{k}\oplus i\mathfrak{k}$ is obtained by transporting to $K_{\mathbb C}$ the decomposition $T_{1}K_{\mathbb{C}}=\mathfrak{k}\oplus i\mathfrak{k}$ by left translation. In the basis $\{X^{1},\dots,X^{n},\frac{\partial}{\partial y^{1}},\dots,\frac{\partial}{\partial y^{n}}\}$ the derivative of $\alpha$ is
\begin{equation}
\begin{pmatrix}1 & 0\\
0 & H
\end{pmatrix}.
\end{equation}
The standard complex structure on $K_{\mathbb{C}}$ is given by the complex tensor $J_{K_{\mathbb{C}}}=\begin{pmatrix}0 & -1\\1 & 0\end{pmatrix}$, from which one may compute the complex tensor
\[
J_{(x,Y)}^{\tau}:=[(\psi_{\tau}^{-1})_{*}]_{(x,Y)}\circ J_{K_{\mathbb{C}}}\circ[(\psi_{\tau})_{*}]_{(x,Y)}.
\]
It is now straightforward to show that the $(1,0)$-tangent space can be expressed as claimed (by, for example, projecting the space of vectors of the form $\begin{pmatrix}X\\ 0\end{pmatrix}$ onto their $(1,0)$-part; note that we introduce the invertible matrix $\frac{e^{i\tau_{2}ad_{u(Y)}} \, ad_{u(Y)}}{2i\mathrm{sin}(\tau_{2}ad_{u(Y)})}$ to obtain the convenient normalization in \eqref{tauforms}). From (\ref{tauvectors}) (and being careful that $ad_{u(Y)}^{*}w=-w\circ ad_{u(y)}$), the expression in (\ref{tauforms}) follows easily. Note that this frame of holomorphic vector fields is obtained by pushing forward by $\psi_\tau^{-1}$ a frame of left-invariant holomorphic vector fields on $K_{\mathbb C}$. 
\end{proof}

\bigskip
Note that the vector fields $\{Z_{\tau}^{i}\}$ are left $K_{\mathbb C}$-invariant holomorphic vector fields and therefore are coordinate vector fields only for abelian groups.

\begin{lem}
\label{properties} $Ad$-invariance of $h$ implies that
\begin{enumerate}
\item $[Y,u(Y)]=0$,
\item $ad_{Y}=H(Y)^{-1}ad_{u(Y)}=ad_{u(Y)}H(Y)^{-1}$, and
\item $u(Y)$ is equivariant, that is $\forall x\in K$, $u(Ad_{x}(Y))=Ad_{x}(u(Y))$.
\end{enumerate}
\end{lem}

\begin{proof}
The three conditions follow easily from $Ad$-invariance of $h$. 
\end{proof}

\begin{lem}
The canonical symplectic form $\omega$ is of type $(1,1)$ with respect to $J^{\tau}$. 
\end{lem}

\begin{proof}
This follows from direct computation from (\ref{eq:sympform}) and (\ref{tauforms}).
\end{proof}

\begin{proof}[of Theorem \ref{kahler}] To prove positivity, one computes
\begin{equation}
i\omega(\bar{Z}_{j}^{\tau},Z_{k}^{\tau})=\left[\frac{ad_{Y}}{2\sin(\tau_{2}ad_{u})}e^{i\tau_{2}ad_{u}}\right]_{k}^{j}.\label{posi}
\end{equation}

Since $ad_{Y}=ad_{u}\cdot H^{-1}$ and $[ad_{u},H^{-1}]=0$, the symmetric matrix $H^{-1}(Y)$ and the antisymmetric matrix $ad_{u}(Y)$ can be simultaneously diagonalized. Let $\beta>0$ be an eigenvalue of $H^{-1}$ and let $i\lambda,\lambda\in\mathbb{R}$, be an eigenvalue of $ad_{u}$ for the same eigenvector. Then, substituting in to (\ref{posi}), we get for this eigenspace that
\[
\frac{e^{-\lambda\tau_{2}}}{\tau_{2}}\beta\frac{\lambda\tau_{2}}{e^{\lambda\tau_{2}}-e^{-\lambda\tau_{2}}}
\]
This is positive as long as $\tau_{2}>0$.

To determine the K\"ahler potential we follow a calculation similar to the one in \cite{Hall-Kirwin}. We have
\[
\theta(Z_{j}^{\tau})=\left(\frac{1}{2}+\frac{i\tau_{1}}{2\tau_{2}}\right)y^{j}.
\]
On the other hand,
\[
Z_{j}^{\tau}\left(\sum_{i=1}^{n}y^{i}u^{i}-h\right)=-\frac{i}{2\tau_{2}}y^{j},
\]
so that
\[
\theta(Z_{j}^{\tau})=-\bar{\tau}Z_{j}^{\tau}\left(\sum_{k=1}^{n}y^{k}u^{k}-h\right).
\]
Therefore, for any vector field pointwise of type $(1,0)$ the same equation holds. In particular, taking holomorphic coordinate vector fields we obtain
\[
\theta^{(1,0)}=\partial(-\bar{\tau}(Yu-h)),
\]
and it follows that
\[
\omega=-d\theta=-(\partial+\bar{\partial})\theta=i\partial\bar{\partial}(2\tau_{2}(Yu-h))=\omega^{(1,1)},
\]
and the K\"ahler potential is $\kappa=2\tau_{2}(Yu-h)$.
\end{proof}

\section{Time evolution and change of polarization\label{sec:Time-evolution}}

As in the introduction, let $\frac{1}{2}\kappa(x,Y)=\frac{1}{2}|Y|^{2}$ be the kinetic energy function on $T^{*}K$. The Hamiltonian flow $\Phi_{\sigma}^{\kappa/2}:T^{*}K\rightarrow T^{*}K$ corresponding to $\kappa/2$ is the geodesic flow for the bi-invariant metric defined by $B$ on $K$. The natural complex structure on $T^{*}K$ induced by pulling back the complex structure from $K_{\mathbb{C}}$ by the map $(x,Y)\in T^{*}K\mapsto xe^{iY}\in K_{\mathbb{C}}$ can be understood as the analytic continuation of this geodesic flow to ``time-$i$'' in several ways (see \cite{Hall-Kirwin,Guillemin-Stenzel-91,Guillemin-Stenzel-92,Lempert-Szoke-91,Szoke-91} for more details): first, the pushforward by the time-$\sigma$ geodesic flow of the vertical tangent space yields a family of distributions depending on $\sigma\in\mathbb{R}$ which can be analytically continued to yield a complex Lagrangian distribution which turns out to be the $(1,0)$-tangent bundle of $T^{*}K$. Second, if $f$ is a real-analytic function on $K$ which admits an analytic continuation to $T^{*}K$ (with respect to the canonical complex structure), then the value of the analytic continuation of $f$ at $(x,Y)$ can be computed as the analytic continuation of the function $\sigma\mapsto f\circ\pi\circ\Phi_{\sigma}^{\kappa/2}(x,Y)$ to $\sigma=i$. In \cite{Hall-Kirwin2}, Hall and the first author showed that this construction can be generalized to magnetic flows (i.e. Hamiltonian flows of $\kappa/2$ with respect to twisted symplectic forms). In both cases, it is clear that one obtains positive complex (i.e. K\"ahler) structures not only from the time-$i$ flow, but from the time-$\tau$ flow for any $\tau\in\mathbb{C}^{+}.$

In this section, we begin by showing that the family of complex structures described in the previous section arise as the ``time-$\tau$'', $\tau\in\mathbb{C}^{+}$, Hamiltonian flow of the functions $h$ with respect to the standard symplectic form on $T^{*}K$, and we describe this phenomenon from various points of view (in terms of the polarization and in terms of holomorphic functions and trivialized sections). We will also discuss the time-$t$ flow for real $t$. Next, we will study time-$\tau$ $h$-holomorphic half-forms and the resulting K\"ahler quantization of $T^{*}K.$

\bigskip{}
For $\tau\in\mathbb{C}^{+}$, the complex structure $J^{\tau}$ on $T^{*}K$ described in the previous section can be understood in terms of the ``time-$\tau$'' Hamiltonian flow of $h$ at the level of holomorphic functions as in the following theorem, which generalizes the analytic continuation formulas of \cite{Thiemann96,Hall02,Hall-Kirwin}.

Let us prove the following lemma, which describes the Hamiltonian vector field $X_{h}$.

\begin{lem}
The Hamiltonian vector field of the function $h(Y)$ is given by
\begin{equation}
X_{h}=\sum_{i=1}^{n}u^{i}X_{i},\label{hamvect}
\end{equation}
where, as above, $u^{i}=\frac{\partial h}{\partial y^{i}},i=1,\dots,n$. Moreover, $X_{h}$ is a complete vector field on $T^{*}K$, with flow given explicitly by
\[
\phi_{h}(t;(x,Y))=(xe^{tu},Y),\ t\in\mathbb{R}.
\]
\end{lem}

\begin{proof}
By direct computation, using (\ref{eq:sympform}), the fact that $[Y,u(Y)]=0$ due to $Ad$ invariance of $h$, and $\iota_{X_{h}}\omega=dh=\sum_{i=1}^{n}u^{i}dy^{i}$, one obtains (\ref{hamvect}). It is immediate to check the expression for the flow $\phi$.
\end{proof}

\begin{thm}
\label{analyticcont} Suppose $f\in C^{\infty}(K)$ admits a $J^{\tau}$-analytic continuation (also denoted by $f$) to $T^{*}K$ for some $\tau\in\mathbb{C}^{+}$.
Then,
\[
f(xe^{\tau u})=e^{\tau X_{h}}f(x):=\sum_{_{k=0}}^{\infty}\frac{\tau^{k}X_{h}^{k}}{k!}f(x).
\]
Explicitly, if $f$ is given in Peter--Weyl form
\[
f(x)=\sum_{\rho\in\hat{K}}Tr(A_{\rho}\rho(x)),
\]
where $A_{\rho}$ is an endomorphism of the representation space $V_{\rho}$, then
\[
e^{\tau X_{h}}f(x)=\sum_{\rho\in\hat{K}}Tr(A_{\rho}\rho(xe^{\tau u})),
\]
where $\rho(xe^{\tau u})$ is the unique representation of $K_{\mathbb{C}}$ whose restriction to $K$ is $\rho$. 
\end{thm}

\begin{proof}
Let $\tilde{J}^{\tau}$ be the complex structure on $T^{*}K$ obtained by pulling back the canonical complex structure on $K_{\mathbb{C}}$ by the diffeomorphism $\psi_{\tau}$, so that $J^{\tau}=\alpha_{*}^{-1}\circ\tilde{J}^{\tau}\circ\alpha_{*}$. From \cite[Thm 14]{Hall02} and \cite{Hall-Kirwin}, we know that
\[
e^{\tau\sum_{j=1}^{n}u^{j}X_{j}}f(x)
\]
is the $\tilde{J}^{\tau}$-analytic continuation of $f$ to $T^{*}K$. Therefore, composing with $\alpha$ gives the $J_{\tau}$-analytic continuation of $f$ to $T^{*}K$. But, since $\alpha_{*}X_{j}=X_{j},\ \alpha^{*}u=u(Y)$, we have
\[
\alpha^{*}\left(e^{\tau\sum_{j=1}^{n}u^{j}X_{j}}f(x)\right)=e^{\tau X_{h}}f(x),
\]
as desired.
\end{proof}

\bigskip{}
Since the function $h$ is not necessarily real analytic, it is perhaps not clear that the (real-time) Hamiltonian flow of $h$ can be expressed as a convergent exponential series of $tX_{h}$ (which is known to be true in the case for $h(Y)=\frac12 |Y|^2$), but it is nevertheless true, as the following lemma shows.

\begin{cor}
\label{realanalyticcont} Let $f\in C^{\infty}(K)$ admit an analytic continuation (also denoted by $f$) to $T^{*}K$, as in Theorem \ref{analyticcont}. Then we can take $t\in\mathbb{R}\in\partial\mathbb{C}^{+}$, and the action of the Hamiltonian flow of $h$ on functions can be written as a convergent series as
\begin{equation}
f(xe^{tu})=e^{tX_{h}}f(x)=\sum_{_{k=0}}^{\infty}\frac{t^{k}X_{h}^{k}}{k!}f(x).\label{eq:uTaylorSeries}
\end{equation}
\end{cor}

\begin{rem}
In the case that $h=\frac12 |Y|^2$ is the usual quadratic kinetic energy function, a short computation shows that $X_h^{k}f(x)$ is a homogeneous polynomial in $Y$ of degree $k$, so that \eqref{eq:uTaylorSeries} can be regarded as a Taylor series in the fiber. For more general $h$, $X_h^{k}f(x)$ is no longer necessarily a homogeneous polynomial in $Y$.
\end{rem} 

\begin{proof} This follows by evaluating the formulae in Theorem \ref{analyticcont} at $\tau=t\in\mathbb{R}$. Note that the convergence of the Peter--Weyl series for $\tau\in\mathbb{C}^{+}$ implies its convergence for $\tau=t\in\mathbb{R}$, since for $\tau_{2}>0,$ the operator norms satisfy
\[
||\rho(xe^{tu}))||\leq||\rho(xe^{(t+i\tau_{2})u}))||.
\]
\end{proof}

\bigskip
Recall that the vertical polarization $\mathcal{P}^{0}$ is the complexification of the vertical tangent space to $T^{*}K$.

\begin{thm}
\label{thmomegatau} Let $\tau\in\mathbb{C}^{+}$. Then
\begin{enumerate}
\item interpreted as an infinite series, the operator $e^{\bar{\tau}\mathcal{L}_{X_{h}}}$ applied to the section $\frac{\partial}{\partial u^{i}}$ of $\mathcal{P}^{0}$ converges, and hence
\[
\mathcal{P}^{\tau}=e^{\bar{\tau}\mathcal{L}_{X_{h}}}\mathcal{P}^{0},\ \tau\in\mathbb{C}^{+}
\]
as distributions, and
\item let $\Omega_{\tau}:=\Omega_{\tau}^{1}\wedge\cdots\wedge\Omega_{\tau}^{n}$ denote the left $K_{\mathbb C}$-invariant trivializing section of the canonical bundle $\bigwedge^{n}(\mathcal{P}^{\tau})^{*}$ corresponding to $\mathcal{P}^{\tau}$, then
\begin{align}
\label{lifooo}
\Omega_{\tau} & =e^{\tau\mathcal{L}_{X_{h}}}\Omega_{0},
\end{align}
where $\Omega_{0}:=w^{1}\wedge\cdots\wedge w^{n}$ is the Haar measure on $K$ pulled back to $T^{*}K$. Moreover, the expression in the right-hand side of (\ref{lifooo}) makes sense as a convergent series.
\end{enumerate}
\end{thm}

\begin{proof}[of Theorem \ref{thmomegatau}] First, using $\mathcal{L}_{x_{h}}\frac{\partial}{\partial u^{j}}=[X_{h},\frac{\partial}{\partial u^{j}}]=-X_{j}$ and $\mathcal{L}_{X_{h}}X_{j}= ad_{u}X_{j}$, one computes that the series $e^{\bar{\tau}\mathcal{L}_{X_{h}}}\frac{\partial}{\partial u^{j}}$ is
\[
e^{\bar{\tau}\mathcal{L}_{X_{h}}}\frac{\partial}{\partial u^{j}}=\frac{\partial}{\partial u^{j}}+(e^{-\bar{\tau}ad_{u(Y)}}-1)X_{j}
\]
is convergent. Hence, $e^{\bar{\tau}\mathcal{L}_{X_{h}}}\mathcal{P}^{0}$ is a well-defined distribution.

Since $\partial/\partial y^{j}= \sum_{k=1}^n H(Y)_{jk}\partial/\partial u^{k}$, to prove (1) we only need to show that, for any $j=1,\dots n$,
\[
\left(e^{\tau\mathcal{L}_{X_{h}}}\frac{\partial}{\partial u^{j}}\right)f(xe^{\tau u})=0
\]
for any $J^{\tau}$-holomorphic function $f$. But, as differential operators on functions,
\[
\left(e^{\tau\mathcal{L}_{X_{h}}}\frac{\partial}{\partial u^{j}}\right)=e^{\tau X_{h}}\frac{\partial}{\partial u^{j}}e^{-\tau X_{h}},
\]
which, since $f(xe^{\tau u})=e^{\tau X_{h}}f(x)$ and $\frac{\partial}{\partial u^{j}}f(x)=0$, proves the formula in (1).

To prove (2), note that from (\ref{tauforms}) and using $du^{j}=Hdy^{j}$, we have
\[
\frac{\partial}{\partial\tau}\Omega_{\tau}^{j}=ad_{u}^{*}\Omega_{\tau}^{j}+du^{j}.
\]
On the other hand, Cartan's formula and the Maurer--Cartan equation tell us that
\begin{gather*}
(\mathcal{L}_{X_{h}}w^{j})(X_{k})=d(w^{j}(X_{h}))(X_{k})+dw^{j}(X_{h},X_{k})\\
=du^{j}(X_{k})-\sum_{l,m}C_{lm}^{j}w^{l}\wedge w^{m}(u^{r}X_{r},X_{k})=-u^{r}C_{rk}^{j}
\end{gather*}
and that
\[
(\mathcal{L}_{X_{h}}w^{j})\left(\frac{\partial}{\partial y^{k}}\right)=du^{j}\left(\frac{\partial}{\partial y^{k}}\right)+dw^{j}\left(X_{h},\frac{\partial}{\partial y^{k}}\right)=\frac{\partial}{\partial y^{k}}(u^{j})=H_{k}^{j}
\]
whence
\[
\mathcal{L}_{X_{h}}w^{j}=ad_{u}^{*}w^{j}+du^{j}=\frac{\partial}{\partial\tau}\Omega_{\tau}^{j}.
\]
Since $\Omega_{\tau=0}^{j}=w^{j}$, it follows that $\Omega_{\tau}^{j}=e^{\tau\mathcal{L}_{X_{h}}}w^{j}$, and hence that $\Omega_{\tau}:=\Omega_{\tau}^{1}\wedge\cdots\wedge\Omega_{\tau}^{n}=e^{\tau\mathcal{L}_{x_{h}}}(w^{1}\wedge\cdots\wedge w^{n})$ as desired. From (\ref{tauforms}) it follows that $\Omega_\tau$ is analytic in $\tau$ and, in particular, the expression for $\Omega_\tau$, at every point $(x,Y)$, makes sense as a convergent analytic power series for $\tau \in \mathbb{C}^+$.
\end{proof}

\begin{cor}
The $(1,0)$-tangent bundle $\mathcal{P}^{\tau}$ of the complex structure $J^{\tau}$ is the analytic continuation to time-$(\sigma=-\bar{\tau})$ of the pushforward of $\mathcal{P}^{0}$ by the time-$\sigma$ Hamiltonian flow $\Phi_{\sigma}^{h}$ of $h$.
\end{cor}

\begin{proof}
Since for any (real) vector field $X$,
\[
\frac{d}{dt}\Big\vert_{t=0}(\Phi_{t}^{h})_{*}X=[X,X_{h}]
\]
we see that $(\Phi_{t}^{h})_{*}X=e^{-t\mathcal{L}_{X_{h}}}X.$ The result now follows from Theorem \ref{thmomegatau}(1).
\end{proof}

\bigskip{}
As $\tau$ approaches the real boundary of $\mathbb{C}^{+}$, i.e. as $\tau_{2}\to0$, the polarizations $\mathcal{P}^{\tau}$ become real.

\begin{prop}
\label{taupol} The Hamiltonian function $h$ defines a family of real polarizations on $T^{*}K$
\[
\mathcal{P}^{t}=e^{t\mathcal{L}_{X_{h}}}\mathcal{P}^{0},\ t\in\mathbb{R}\subset\partial\mathbb{C}^{+}.
\]
\end{prop}

\begin{proof}
As in the previous corollary, the operator $e^{t\mathcal{L}_{X_{h}}}$ applied to a vector field is the pushforward by the time $t$ Hamiltonian flow of $X_{h}$. This implies involutivity and also that the resulting distribution is Lagrangian.
\end{proof}

\bigskip
We turn now to half-forms and the resulting half-form-corrected K\"ahler quantization of $T^{*}K$.

For $\tau\in\mathbb{C}^{+}$, let
\[
\beta_{\tau}(Y)=\pi^{-n/4}e^{i\tau(B(u(Y),Y)-h(Y))}
\]
and define a measure $d\mu_{\tau}$ on $T^{*}K$ by
\begin{equation}
d\mu_{\tau}:=\beta_{\tau}\bar{\beta}_{\tau}|\sqrt{\Omega_{\tau}}|^{2}\frac{\omega^{n}}{n!}=e^{-\kappa(Y)}\tau_{2}^{\frac{n}{2}}\eta(\tau_{2}u)(\det H)^{\frac{1}{2}}\frac{\omega^{n}}{\pi^{n/2}n!}.\label{eq:dmu}
\end{equation}
The second equality above is a consequence of the following computation.

\begin{lem}
\label{lem:BKSnorm}For $\tau\in\mathbb{C}^{+}$,
\begin{equation}
|\sqrt{\Omega_{\tau}}|^{2}=\sqrt{\frac{\bar{\Omega}_{\tau}\wedge\Omega_{\tau}}{(2i)^{n}(-1)^{n(n-1)/2}\omega^{n}/n!}}=\tau_{2}^{\frac{n}{2}}\eta(\tau_{2}u)(\det H)^{\frac{1}{2}}.\label{halfnorm}
\end{equation}
\end{lem}

\begin{proof}
The result follows from direct computation. As an alternative, after checking that the result is independent of $\tau_{1}$, which follows from Theorem \ref{thmomegatau}(2), one can use Proposition 2 of \cite{Florentino-Matias-Mourao-Nunes05} and pull back the result by the holomorphic (but not symplectic) diffeomorphism $\alpha$.
\end{proof}

\bigskip{}
In Section \ref{gcsts} we will need the following lemma.

\begin{lem}
\label{adinvmeasure} Let $\tau\in\mathbb{C}^{+}$. The measure $d\mu_{\tau}$ is $Ad_{K}$-invariant.
\end{lem}

\begin{proof}
From Lemma \ref{properties}, differentiating once more with respect to $y$, one obtains $Ad_{K}$-equivariance of $H$, that is $\forall g\in K,\ H(Ad_{g}Y)=Ad_{g}H(Y)$, where here we interpret $H$ as a linear operator $\mathfrak{k}\to\mathfrak{k}.$ Therefore, $\det H$ is $Ad_{K}$-invariant. On the other hand, the equivariance of $u(Y)$ gives that $\kappa_{\tau}(Y)$ is also $Ad_{K}$-invariant.
\end{proof}

\bigskip
Let us now examine the $\mathcal{P}^{\tau}$-polarized sections of the prequantum line bundle $L\rightarrow T^{*}K$. Recall that we have trivialized $L$ with respect to the symplectic potential $\theta$. In particular, sections of $L$ are (complex-valued) functions on $T^{*}K\simeq K\times\mathfrak{k}$ and the covariant derivative is given by $\nabla^{L}=d+i\theta$.

\begin{thm}
\label{tautwopositive} Let $\tau\in\mathbb{C}^{+}$. Then a trivialized section of $L\otimes\sqrt{\mathcal{K}^{\mathcal{P^{\tau}}}}$ is covariantly constant along the polarization $\overline{\mathcal{P}^{\tau}}$ if and only if it is of the form
\[
f(xe^{\tau u})\beta_{\tau}(Y)\otimes\sqrt{\Omega_{\tau}},
\]
for some $J^{\tau}$-holomorphic function $f$. Moreover, if $\rho\in\hat{K}$ is an irreducible representation of $K$ and $A$ is an endomorphism of $V_{\rho}$, Then
\[
Tr(A\rho(xe^{\tau u}))
\]
is square integrable with respect to $d\mu_{\tau}$. 
\end{thm}

\begin{proof}
A trivialized section $\sigma\in\Gamma(L)\simeq C^{\infty}(K\otimes\mathfrak{k})$ is covariantly constant along $\overline{\mathcal{P}^{\tau}}$ if and only if for $j=1,...,n$,
\[
\bar{Z}_{j}^{\tau}\sigma+i\theta(\bar{Z}_{j}^{\tau})\sigma=0.
\]
From the proof of Theorem \ref{kahler}, we have $\theta(\bar{Z}_{j}^{\tau})=-\tau\bar{Z}_{j}^{\tau}(B(u(Y),Y)-h(Y))$, so that the equations of covariant constancy have solutions of the form
\[
f(xe^{\tau u})\beta_{\tau}(Y),
\]
with $f$ holomorphic. This proves the first part. To prove the second claim, recall that (see, for example, \cite{Hall94}), given $\rho\in\hat{K}$ and $A\in\mathrm{End}(V_{\rho})$, there exist ($\rho$-dependent) constants $c_{0},c_{1}>0$ such that
\begin{equation}
|Tr(A\rho(xe^{\tau u}))|\leq c_{o}e^{c_{1}\tau_{2}||u||}.\label{normpi}
\end{equation}
If we change variables in the computation of the $L^{2}$-norm of $Tr(A\rho(xe^{\tau u}))$ with respect to $d\mu_{\tau}$, an additional factor of $(\det H)^{-1}$ appears from the Jacobian. Together with the $(\det H)^{\frac{1}{2}}$ factor in $d\mu_{\tau}$, this gives a bounded function $(\det H)^{-\frac{1}{2}}$ which does not affect integrability. On the other hand, let $a(u)=2(B(u(Y),Y)-h(Y(u)))$, so that $\frac{\partial a}{\partial u^{j}}=2y^{j}$. Let $Y_{0}$ be the unique value of $Y$ such that $u(Y_{0})=0$. One has $a(0)=-2h(Y_{0})$ and
\begin{align*}
a(u)&=a(0)+2\int_{0}^{1}B(u(Y),Y(tu))\,dt\\
&=a(0)+2B(u(Y),Y_{0})+2\int_{0}^{1}\int_{0}^{1}B(u(Y),H^{-1}(stu)ut)\, dtds.
\end{align*}
Let $\lambda_{m}(stu)>0$ be the minimum eigenvalue of the symmetric positive definite matrix $H^{-1}(stu)$. We then have
\[
a(u)\geq a(0)+2B(u(Y),Y_{0})+2\left(\int_{0}^{1}\int_{0}^{1}\lambda_{m}(stu)tdtds\right)||u||^{2}.
\]
Therefore, there exist constants $b_{0},b_{1}>0$ such that
\[
e^{-\kappa(Y)}=e^{-\tau_{2}a(u)}\leq b_{0}e^{-b_{1}\tau_{2}||u||^{2}}.
\]
This, together with (\ref{normpi}), implies that $Tr(A\rho(xe^{\tau u}))\in L^{2}(T^{*}K,d\mu_{\tau})$.
\end{proof}

\bigskip
Following the usual prescription of half-form quantization with respect to a K\"ahler polarization, we have the following definition. 

\begin{defn}
\label{hilberts} Let $\tau\in\mathbb{C}^{+}$. The Hilbert space $\mathcal{H}_{\tau}\cong L_{\mbox{hol}}^{2}(T^{*}K,d\mu_{\tau})$ of $\mathcal{P}^{\tau}$-polarized square-integrable half-form corrected sections is the norm completion of the space of finite linear combinations of sections of the form 
\begin{equation}
Tr(A\rho(xe^{\tau u}))\beta_{\tau}(Y)\otimes\sqrt{\Omega_{\tau}},\label{wave}
\end{equation}
where $\rho\in\hat{K}$ and $A$ is an endomorphism of $V_{\rho}$. The inner product of two such sections is given by
\begin{eqnarray}
 & \langle Tr(A\rho(xe^{\tau u}))\beta_{\tau}(Y)\otimes\sqrt{\Omega_{\tau}},Tr(A^{\prime}\rho^{\prime}(xe^{\tau u})))\beta_{\tau}(Y)\otimes\sqrt{\Omega_{\tau}}\rangle_{\tau}=\notag\label{innerprod}\\
 & \int_{T^{*}K}\overline{Tr(A\rho(xe^{\tau u}))}Tr(A^{\prime}\rho^{\prime}(xe^{\tau u}))d\mu_{\tau}.
\end{eqnarray}
\end{defn}

\begin{rem}
We will show in Proposition \ref{admissible} that the wave functions (i.e. sections) in $\mathcal{H}_{\tau}$ corresponding to matrix elements of irreducible representations of $K$ actually form an orthogonal basis.
\end{rem}

\bigskip{}
Let us now describe the Hilbert spaces of polarized sections for the polarizations $\mathcal{P}^{t}$, $t\in\mathbb{R}\subset\partial\mathbb{C}^{+}$. Notice that solutions of the equations of covariant constancy are still of the form (\ref{wave}), but now with $t\in\mathbb{R}\subset\partial\mathbb{C}^{+}$. Since we are using half-forms, the BKS pairing (defined for nontransverse polarizations in \cite[(p. 187,231)]{Woodhouse}) induces a canonical inner product on the spaces of $\mathcal{P}^{t}$-polarized sections of $L\otimes\sqrt{\mathcal{K}^{\mathcal{P}^{t}}}$. We will now show that the inner product structure on these spaces of covariantly constant sections coincides with the inner product defined by continuity from the inner products of $\mathcal{H}_{\tau},\ \tau\in\mathbb{C}^{+}$. As we will see, this will allow us to smoothly extend the quantum Hilbert bundle $\mathcal{H}\rightarrow\mathbb{C}^{+}$ (whose fiber at $\tau\in\mathbb{C}^{+}$ is $\mathcal{H}_{\tau}$) to $\mathbb{C}^{+}\cup\mathbb{R}$ by attaching the real-polarized quantum Hilbert spaces to the points in $\mathbb{R}\subset\partial\mathbb{C}^{+}$.

We describe now the limit inner product in $\mathcal{H}_{t\in\mathbb{R}}$ directly in terms of matrix elements of irreducible representations of $K$. 
\begin{prop}
\label{real} Let $\tau\in\mathbb{C}^{+}$, $\rho,\rho^{\prime}\in\hat{K}$, $i,j=1,\dots,\dim\rho,\ k,l=1,\dots,\dim\rho^{\prime}$ and
\[
\sigma_{\tau}=\rho(xe^{\tau u})_{ij}\beta_{\tau}(Y)\otimes\sqrt{\Omega_{\tau}},\sigma_{\tau}^{\prime}=\rho^{\prime}(xe^{\tau u})_{kl}\beta_{\tau}(Y)\otimes\sqrt{\Omega_{\tau}}.
\]
Then,
\begin{enumerate}
\item the inner product $\langle\sigma_{\tau},\sigma_{\tau}^{\prime}\rangle_{\tau}$
is independent of $\tau_{1}$,
\item the inner product $\langle\sigma_{\tau},\sigma_{\tau}^{\prime}\rangle_{\tau}$
has a finite limit as $\tau_{2}\to0$, and
\item the limit in (2) is given by the inner product defined on the Hilbert
space of states for the vertical polarization ($\tau=0$) \cite{Hall94,Florentino-Matias-Mourao-Nunes06}
\[
\lim_{\tau_{2}\to0}\langle\sigma_{\tau},\sigma_{\tau}^{\prime}\rangle_{\tau}=\langle\rho(x)_{ij}\otimes\sqrt{\Omega_{0}},\rho^{\prime}(x)_{kl}\otimes\sqrt{\Omega_{0}}\rangle_{0}=\langle\rho(x)_{ij},\rho^{\prime}(x)_{kl}\rangle_{L^{2}(K,dx)}=\frac{\delta_{\rho\rho^{\prime}}}{\dim\rho}\delta_{ik}\delta_{jl}.
\]

\end{enumerate}
\end{prop}
\begin{proof}
To prove (1), let us perform the integration over $K$ in (\ref{innerprod}), where we take $TrA\rho=\rho_{ij},\ TrA^{\prime}\rho^{\prime}=\rho_{kl}^{\prime}$. We have, from Weyl's orthogonality conditions,
\[
\int_{K}\overline{\rho_{ia}(x)}\rho_{kb}^{\prime}(x)dx=\frac{\delta_{\rho\rho^{\prime}}}{d_{\rho}}\delta_{ab}\delta_{ik}.
\]
Therefore, in the integral over the fibers the only $\tau_{1}$ dependence is in the factor
\[
\overline{\rho_{aj}(e^{\tau u})}\rho_{al}(e^{\tau u})=\sum_{r=1}^{d_{\rho}}\overline{\rho_{rj}(e^{i\tau_{2}u})}\rho_{rl}(e^{i\tau_{2}u})=\rho_{jl}(e^{2i\tau_{2}u}),
\]
since $\rho$ is unitary. This proves that $\langle\sigma_{\tau},\sigma_{\tau}^{\prime}\rangle_{\tau}$ is independent of $\tau_{1}$.

To prove (2), recall from the proof of Theorem \ref{tautwopositive} that
\[
e^{-\kappa(Y)}\leq b_{0}e^{-b_{1}\tau_{2}||u||^{2}},
\]
so that the norm of $\sigma_{\tau}$ for a given choice of Hamiltonian function $h$, satisfying the above conditions, is bounded by the norm of the corresponding $\sigma_{b_{1}\tau}$ in the case $h(Y)=\frac{1}{2}|Y|^{2}$. Since in this case the norm of $\sigma_{\tau}$ is bounded as $\tau_{2}\to0$ (see \cite{Hall94,Florentino-Matias-Mourao-Nunes06}) the result follows. To prove (3), let $a>0$ and consider the integral
\[
I(a,\tau_{2})=\pi^{-\frac{n}{2}}\int_{\mathfrak{k}}\rho_{jl}(e^{2iau})e^{-2a^{2}\frac{B(u,Y)-h}{\tau_{2}}}\tau_{2}^{\frac{n}{2}}\eta(\tau_{2}u)(\det H)^{\frac{1}{2}}dY.
\]
After changing variables $\tilde{u}=au$ and with $u=aY$, $\tilde{h}=a^{2}h$, we obtain
\[
I(a,\tau_{2})=\pi^{-\frac{n}{2}}\int_{\mathfrak{k}}\rho_{jl}(e^{2i\tilde{u}})e^{-2\frac{B(u,\tilde{u})-\tilde{h}}{\tau_{2}}}\tau_{2}^{\frac{n}{2}}\eta(\tilde{u})(\det H)^{-\frac{1}{2}}a^{-n}d\tilde{u}=a^{-n}I(1,\tau_{2}).
\]
In the limit $\tau_{2}\to0$, $I(1,\tau_{2})$ can be evaluated by Laplace's approximation. The minimum of $B(u,\tilde{u})-\tilde{h}$ is at $\tilde{u}_{0}$, where $u(\tilde{u}_{0})=0$. The Hessian is given by $H^{-1}(0)$. Laplace's approximation gives, to leading order as $\tau_{2}\to0$,
\[
\lim_{\tau_{2}\to0}I(1,\tau_{2})=\tau_{2}^{n}\rho(e^{2i\tilde{u}_{0}})_{jl}\eta(\tilde{u}_{0}).
\]
Therefore, setting now $a=\tau_{2}\to0$, we obtain $\tilde{u}_{0}\to0$ and $\lim_{\tau_{2}\to0}I(\tau_{2},\tau_{2})=\delta_{jl}$ so that
\[
\lim_{\tau_{2}\to0}\langle\sigma_{\tau},\sigma_{\tau}^{\prime}\rangle_{\tau}=\frac{\delta_{\rho\rho^{\prime}}}{d_{\rho}}\delta_{jl}\delta_{ik},
\]
as desired. Note that the subleading terms in the Laplace approximation \cite{Kirwin08} do not contribute, as they go to zero faster than $\tau_{2}^{n}.$
\end{proof}

\begin{defn}
\label{boundaryhilberts} Let $t\in\mathbb{R}\subset \partial \mathbb{C}^{+}$. The Hilbert space $\mathcal{H}_{t}$ of $\mathcal{P}^{t}$-polarized half-form corrected sections is the norm completion of the space of finite linear combinations of sections of the form
\[
Tr(A\rho(xe^{tu}))\beta_{t}(Y)\otimes\sqrt{\Omega_{t}},
\]
where $\rho\in\hat{K}$ and $A$ is an endomorphism of $V_{\rho}$. The inner product of two such sections is
\[
\langle\rho_{ij}(xe^{tu})\beta_{t}(Y)\otimes\sqrt{\Omega_{t}},\rho_{kl}^{\prime}(xe^{tu})\beta_{t}(Y)\otimes\sqrt{\Omega_{t}}\rangle_{t}=\frac{\delta_{\rho\rho^{\prime}}}{\dim\rho}\delta_{ik}\delta_{jl}.\notag
\]
\end{defn}

It follows that there is a bundle of Hilbert spaces, $\mathcal{H}\to\mathbb{C}^{+}\cup\mathbb{R}$, with fiber $\mathcal{H}_{\tau}$ over $\tau$, with continuous hermitian structure.

\section{Generalized CSTs from geometric quantization}

\label{gcsts}

In this section we will associate to each $h\in C^{\infty}(T^{*}K)$'satisfying properties \eqref{eq:3props} a unitary isomorphism $U_{\tau}:\mathcal{H}_{0}\rightarrow\mathcal{H}_{\tau}$. When $h=E$ is the kinetic energy Hamiltonian on $T^{*}K$ and $\tau=is\in i\mathbb{R}$, the map $U_{is}$ is Hall's CST. To do so, we first use the usual Kostant--Souriau quantization (with half-forms) to define an operator $e^{-i\tau\hat{h}}:\mathcal{H}_{0}\rightarrow\mathcal{H}_{\tau}$. We will show that for all $\tau\in\mathbb{C}^{+}\cup\mathbb{R}$, there is a canonical unitary $K\times K$ action the quantum Hilbert space $\mathcal{H}_{\tau}$, and that $e^{-i\tau\hat{h}}$ intertwines these $K\times K$ actions. Note that for $\tau=t\in\mathbb{R}$, and in general only for such $\tau$, the map $e^{-it\hat{h}}$ is unitary.

Now, suppose that, even though $\hat{h}$ does not preserve $\mathcal{H}_{0}$, we have \emph{some} method of quantizing $h$ to obtain a (possibly unbounded) linear operator $Q(h):\mathcal{H}_{0}\rightarrow\mathcal{H}_{0}$ (for example, if $h$ is a polynomial one might choose an ordering and apply the ``canonical'' quantization $y^{j}\mapsto\hat{y}^{j}$). For real $\tau=t\in\mathbb{R}$, the operator $ e^{itQ(h)}$ is unitary and corresponds to the $-t$ evolution of observables in the Heisenberg picture. Evolving then further in time $+t$ with the geometric quantization polarization changing operator $e^{-it \hat h}$, we see that the operator, $e^{-it \hat h} \circ e^{itQ(h)}$, gives a new representation of the original observables in the Hilbert space corresponding to the time $+t$ polarization. We will consider representation of algebras of observables on $\mathcal{H}_{0}$ and $\mathcal{H}_{\tau}$ in detail in \cite{Kirwin-Mourao-Nunes}.

For complex $\tau$, the operator $e^{-i\tau\hat{h}}\circ e^{i\tau Q(h)}$ remains unitary for $h$ equal to a scalar multiple of $\kappa$, but not in general for other cases. As mentioned above, we will see that for $\tau=is$ and $h=\frac12 |Y|^2$, one can take $E(\tau,h)=e^{i\tau Q(h)}$ with $Q(h)=-\frac{1}{2}\Delta+\frac{\left|\rho\right|^{2}}{2}$ (here, $\rho$ is half the sum of the positive roots of $\mathfrak{k}_{\mathbb{C}}$) to obtain Hall's CST, so that indeed the generalized $h$-CSTs are generalizations of the usual CST.

We will conclude the section by discussing briefly how the existence of the unitary maps $U_{\tau}$ can be explained by Mackey's generalization of the Stone-von Neumann theorem (thus answering a question of Hall in \cite{hall00} for the case $h(Y)=\frac12 |Y|^2$).

\bigskip{}
Recall from the standard results of geometric quantization, that if a (real or complex) function $f$ on $T^{*}K$ preserves a polarization, then its Kostant--Souriau prequantization is the operator 
\begin{equation}
\hat{f}:=\left(i\nabla_{X_{f}}^{L}+f\right)\otimes1+1\otimes\mathcal{L}_{X_{f}}\label{eq:KSquant}
\end{equation}
(acting on sections of $L\otimes\sqrt{\mathcal{K}^{\mathcal{P}}}$) and in fact in fact acts on the Hilbert space of polarized states. If the two functions $f$ and its complex conjugate $\overline f\in C^{\infty}(T^{*}K)$ both preserve the polarization, then the corresponding Kostant--Souriau operators are hermitian conjugate of each other.

For $\tau\in\mathbb{C}^{+}\cup\mathbb{R}$, consider the natural action of $K\times K$ on $\mathcal{H}_{\tau}$ by left and right translations,
\begin{equation}
U_{(g,g^{\prime})}f(xe^{\tau u})\beta_{\tau}(Y)\otimes\sqrt{\Omega_{\tau}}=f(gxe^{\tau u}g^{\prime})\beta_{\tau}(Y)\otimes\sqrt{\Omega_{\tau}},
\end{equation}
for $g,g^{\prime}\in K$ (recall that $\beta_{\tau}(Y)$ is $Ad_{K}$-invariant). As we shall see below, the vector fields generating this action preserve the polarization $\mathcal{P}^{\tau}$ and $\sqrt{\Omega_{\tau}}$.

\begin{prop}
\label{unitary} The operators $U_{(g,g^{\prime})},\ g,g^{\prime}\in K$ are unitary automorphisms of $\mathcal{H}_{\tau}$. 
\end{prop}

\begin{proof}
For $\tau\in\mathbb{C}^{+},$ unitarity follows from Weyl's orthogonality relations and from $Ad_{K}$-invariance of the measure $d\mu_{\tau}$ (Lemma \ref{adinvmeasure}). For $\tau\in\mathbb{R}$, it follows from Weyl's orthogonality relations and from the explicit expression in Definition \ref{boundaryhilberts}.
\end{proof}

\bigskip{}
We begin by noting a few useful technical results. Recall that $\{\tilde y^j\}$ denote the coordinates on the fibers of $T^*K$ corresponding to the frame of right-invariant one-forms $\{\tilde w^j\}$ on $K$, so that $\tilde Y(x,Y)= Ad_x(Y)$.

\begin{lem}
\label{techresults} We have
\begin{enumerate}
\item $\sum_{k=1}^{n}(ad_{u})_{jk}\frac{\partial}{\partial u^{k}}\rho_{ab}(e^{\tau u})=[\rho(T_{j}),\rho(e^{\tau u})]_{ab},\ \forall\rho\in\hat{K},\ j=1\dots,n,\ a,b=1,\dots,\dim\rho$ and
\item $\tilde{X}_{j}(\rho(xe^{\tau u})_{a,b})=\rho(xe^{\tau u}Ad_{x^{-1}}(T_{j}))_{ab},\ \forall\rho\in\hat{K},\ j=1\dots,n,\ a,b=1,\dots,\dim\rho$.
\end{enumerate}
\end{lem}

\begin{proof}
To prove (1), it is enough to expand the exponential in powers of $u$ and use $u=\sum_{i=1}^{n}u^{i}T_{i}$ and $(ad_{u})_{jk}=\sum_{i=1}^{n}u^{i}C_{ikj}$. To prove (2), note that the action of the right-invariant vector field $\tilde{X}_{j}$ in the $(x,\tilde{Y})$ coordinates reads $\tilde{X}_{j}\cdot f(x,\tilde{Y})=\frac{d}{dt}_{|_{t=0}}f(e^{tT_{j}}x,\tilde{Y})$, for any smooth function $f$. 
\end{proof}

\begin{lem}
\label{hats} Suppose $f:\mathfrak{k}\rightarrow\mathbb{C}$ is such that $Y(f(Y))=f(Y)$. Then the Kostant--Souriau prequantum operators associated to $f$ (interpreted now as an $K$-invariant function on $T^{*}K\simeq K\times\mathfrak{k}$) is
\[
\hat{f}=iX_{f}\otimes1+1\otimes i\mathcal{L}_{X_{f}},
\]
where the Hamiltonian vector field $X_{f}$ is
\[
\left(X_{f}\right)_{(x,Y)}=\begin{pmatrix}\mathrm{grad}(f)\\
ad_{Y}\mathrm{grad}(f)
\end{pmatrix}.
\]
\end{lem}

The hypothesis that $Y(f(Y))=f(Y)$ just means that $f$ is linear in the cotangent fibers and zero on the zero section.

\begin{proof}
It is easy to check that the given expression for $X_{f}$ is indeed the Hamiltonian vector field for $f$. In the trivialization $\nabla^{L}=d+i\theta$, one therefore has 
\begin{align*}
\hat{f} & =iX_{f}-\theta(X_{f})+f\\
 & =iX_{f}-Y(f)+f.
\end{align*}
\end{proof}

\bigskip{}
It follows the the Kostant--Souriau prequantum operators associated to the coordinate functions $y^{j}$ and $\tilde{y}^{j}$ are given by
\begin{eqnarray}
\hat{y}^{j} & = & \left(iX_{j}-i\sum_{k=1}^{n}(ad_{Y})_{jk}\frac{\partial}{\partial y^{k}}\right)\otimes1+1\otimes i\mathcal{L}_{X_{y^{j}}},\mbox{ and}\label{yhat}\\
\hat{\tilde{y}}^{j} & = & i\tilde{X}_{j} \otimes1+1\otimes\mathcal{L}_{X_{\tilde{y}^{j}}}.
\end{eqnarray}

\begin{lem}
\label{preserveshalf} The Poisson brackets of $h$ with the coordinate functions $y^{j}$ and $\tilde{y}^{j}$ are zero,
\[
\{h,y^{j}\}=\{h,\tilde{y}^{j}\}=0,\ j=1,\dots,n.
\]
Moreover, $\mathcal{L}_{X_{y^{j}}}\Omega_{\tau}=\mathcal{L}_{X_{\tilde{y}^{j}}}\Omega_{\tau}=0,\ j=1,\dots,n$. 
\end{lem}

\begin{proof}
Recall that the Hamiltonian vector field for the function $h$ is $X_{h}=\sum_{i=1}^{n}u^{i}X_{i}$. Thus, from the proof of Lemma \ref{hats}, the Lie bracket
\[
[X_{h},X_{y^{j}}]=[X_{h},X_{j}-\sum_{k=1}^{n}(ad_{Y})_{jk}\frac{\partial}{\partial y^{k}}]=\sum_{i,k=1}^{n}\left(C_{jk}^i u^{i}X_{k}+(ad_{Y})_{jk}H_{ik}X_{i}\right)=0,
\]
since the structure constants are completely antisymmetric and $(ad_{Y})\cdot H=(ad_{u})$. Therefore, the Lie derivatives $\mathcal{L}_{X_{y^{j}}}$ and $\mathcal{L}_{X_{h}}$ commute. Then the $K\times K$ invariance of $\Omega_{0}$ and Theorem \ref{thmomegatau}(2) imply $\mathcal{L}_{X_{y^{j}}}\Omega_{\tau}=0$. Likewise, defining $\tilde{u}^{i}=\frac{\partial h}{\partial\tilde{y}^{i}}$, one obtains also $X_{h}=\sum_{i=1}^{n}\tilde{u}^{i}\tilde{X}_{i}$ and the same proof applies for $\tilde{y}^{j}$.
\end{proof}

\bigskip{}
The operators $\hat{y}^{j}$ and $\hat{\tilde{y}}^{j}$preserve $\mathcal{H}_{\tau}$ and, as the next theorem shows, generate the natural action of $K\times K$ on $\mathcal{P}^{\tau}$-holomorphic functions on $T^{*}K$ by left and right translations. 

\begin{thm}
\label{preserves} Let $\tau\in\mathbb{C}^{+}\cup\mathbb{R}$. The action of $K\times K$ on $\mathcal{H}_{_{\tau}}$ is generated by the operators $\{\hat{\tilde{y}}^{j},\hat{y}^{j}\}_{j=1,\dots,n}$, where the operators $\{y^{j}\}$ generate the right $K$ action and the operators $\{\hat{\tilde{y}}^{j}\}$ generate the left $K$ action.
\end{thm} 

\begin{proof}
Let $\rho\in\hat{K}$ and let $a,b=1,\dots,\dim\rho$. Then, from Lemma \ref{techresults} and (\ref{yhat}),
\[
-i\hat{y}^{j}\rho(xe^{\tau u})_{ab}=\rho(xT_{j}e^{\tau u})_{ab}-\rho(x[T_{j},e^{\tau u}])_{ab}=\rho(xe^{\tau u}T_{j})_{ab}=\frac{d}{dt}_{|_{t=0}}U_{(1,e^{tT_{j}})}\rho(xe^{\tau u})_{ab}.
\]
On the other hand, is is easy to check that $X_{y^{j}}(\beta_{\tau}(Y))=0$. Moreover, from Lemma \ref{preserveshalf}, we see that the Lie derivatives along $X_{y^{j}}$ also act trivially on $\sqrt{\Omega_{\tau}}$. For $\hat{\tilde{y}}^{j}$, from Lemma \ref{techresults} and (\ref{yhat}) we get
\begin{align*}
-i\hat{\tilde{y}}^{j}\rho(xe^{\tau u})_{ab} & =\rho(xe^{\tau u}x^{-1}T_{j}x)_{ab}+\sum_{i=1}^{n}(Ad_{x^{-1}})_{ij}\rho(x[T_{i},e^{\tau u}])_{ab}\\
 & =\rho(T_{j}xe^{\tau}u)=\frac{d}{dt}_{|_{t=0}}U_{(e^{tT_{j}},1)}\rho(xe^{\tau u})_{ab}.
\end{align*}
Similarly, the action on $\beta_{\tau}$ and $\Omega_{\tau}$ is trivial.
\end{proof}

\bigskip{}
Now, suppose $h$ is a complexifier function on $T^{*}K$ satisfying properties \eqref{eq:3props}. One computes that
\[
e^{-i\tau\hat{h}}=e^{i\tau(B(u(Y),Y)-h(Y))}e^{\tau X_{h}}\otimes e^{\tau\mathcal{L}_{X_{h}}}.
\]

\begin{thm}
\label{evolu} Let $\tau\in\mathbb{C}^{+}\cup\mathbb{R}$. Then $e^{-i\tau\hat{h}}$ is a densely defined linear map from $\mathcal{H}_{0}$ to $\mathcal{H}_{\tau}$ and intertwines the canonical actions of $K\times K$ on $\mathcal{H}_{0}$ and $\mathcal{H}_{\tau}$ by left and right translations. If $t\in\mathbb{R}$, then $e^{it\hat{h}}$ is unitary. 
\end{thm}

\begin{proof}
From Theorem \ref{analyticcont}, Corollary \ref{realanalyticcont}, Definitions \ref{hilberts} and \ref{boundaryhilberts} and Theorem \ref{thmomegatau}(2), it is immediate to see that $e^{-i\tau\hat{h}}$ is a densely defined linear map from $\mathcal{H}_{0}$ to $\mathcal{H}_{\tau}$. From Lemma \ref{preserveshalf} and from the fact that the operators $\hat{y}^{j},\hat{\tilde{y}}^{j}$ annihilate $\beta_{\tau}(Y)$, we see that $\hat{h}$ commutes with these operators which, from Theorem \ref{preserves}, generate the $K\times K$ action. Moreover, from Definition \ref{boundaryhilberts} and from the explicit action of $\hat{h}$, following from Corollary \ref{realanalyticcont}, it is immediate to check that for real $t$, $e^{it\hat{h}}$ is unitary.
\end{proof}

\bigskip{}
We can summarize Theorem \ref{evolu}, Theorem \ref{analyticcont}, Corollary \ref{realanalyticcont}, and Theorem \ref{thmomegatau}(2) in the following.

\begin{prop}
The densely defined operator $e^{-i\tau\hat{h}}:\mathcal{H}_{0}\to\mathcal{H}_{\tau},\ \tau\in\mathbb{C}^{+}$ is the analytic continuation of $\mathcal{P}^{0}$-polarized sections of the half-form corrected prequantum bundle, i.e.
\[
e^{-i\tau\hat{h}}(f(x)\otimes\sqrt{\Omega_{0}})=f(xe^{\tau})\beta_{\tau}(Y)\otimes\sqrt{\Omega_{\tau}}.
\]
Moreover the expression on the left-hand side converges as an infinite power series in $\tau$.
\end{prop}

It is clear that $e^{-i\tau\hat{h}}$ is not a unitary operator. In particular, it intertwines the self-adjoint operator of multiplication by a (non-constant) real function $f\in C^{0}(K)$ on $\mathcal{H}_{0}$ with the non-self-adjoint operator of multiplication by the analytic continuation of $f$ on $\mathcal{H}_{\tau}$. Given the explicit structure of the Hilbert spaces $\mathcal{H}_{\tau}$, however, it is easy to correct this lack of unitarity while keeping the intertwining properties for the $K\times K$ actions. We see that both $\mathcal{H}_{0}$ and $\mathcal{H}_{\tau}$ decompose as infinite direct sums of irreducible representations of $\mathfrak{k}\oplus\mathfrak{k}$,
\begin{equation}
\mathcal{H}_{\tau}=\overline{\oplus_{\rho\in\hat{K}}V_{\rho,\rho}^{\tau}},\label{decomp}
\end{equation}
where $V_{\rho,\rho}^{\tau}=\{Tr(A\rho(xe^{\tau u}))\beta_{\tau}(Y)\otimes\sqrt{\Omega_{\tau}},\ A\in\mathrm{End}V_{\rho},\ \rho\in\hat{K}\}$.

\begin{prop}
\label{blockdiag} The operator $e^{-i\tau\hat{h}}$ preserves this decomposition. 
\end{prop}

\begin{proof}
This is an immediate consequence of Theorem \ref{evolu}.
\end{proof}

\bigskip
We can correct the nonunitarity of $e^{-i\tau\hat{h}}$ while preserving the $K\times K$ action by defining $E(\tau,h):\mathcal{H}_{0}\to\mathcal{H}_{0}$ such that the operators
\begin{equation}
U_{\tau}=e^{-i\tau\hat{h}}\circ E(\tau,h),\label{gencst}
\end{equation}
are unitary. In fact, preserving the $K\times K$ action implies that $E(\tau,h)$ is diagonal with respect with the decomposition (\ref{decomp}) and $E(\tau,h)|_{V_{\rho,\rho}^{\tau}}=\lambda_{\rho}id_{V_{\rho,\rho}^{\tau}}$, for some nonzero constants $\{\lambda_{\rho}\}_{\rho\in\hat{K}}$.

\begin{prop}
\label{admissible} The basis $\{\rho(xe^{\tau u})_{ab}\beta_{\tau}(Y)\otimes\sqrt{\Omega_{\tau}}\}_{\rho\in\hat{K},a,b=1,\dots,\dim\rho}$ is an orthogonal basis of $\mathcal{H}_{\tau}$. Moreover, the norms
\begin{equation}
||\rho(xe^{\tau u})_{ab}\beta_{\tau}(Y)\otimes\sqrt{\Omega_{\tau}}||_{\mathcal{H}_{\tau}}=:a_{\rho}(\tau)
\end{equation}
are independent of $a,b$. Note that for real $\tau$, $a_{\rho}(\tau)=\sqrt{\dim\rho}^{-1}.$ 
\end{prop}

\begin{rem}
Note that $a_{\rho}(\tau)$ actually depends only on $\tau_{2}$, in view of Proposition \ref{real}.
\end{rem} 

\begin{proof}
Each of the vector spaces $V_{\rho,\rho}^{\tau}\subset\mathcal{H}_{\tau},\ \rho\in\hat{K}$, forms an irreducible representation for the compact group $K\times K$. Hence, on each $V_{\rho,\rho}^{\tau}$ there is a unique-up-to-scale $K\times K$ invariant inner product. Therefore the inner product on $\mathcal{H}_{\tau}$ defined by $d\mu_{\tau}$ must be equal up to scale to the inner product on the vertically polarized Hilbert space $\mathcal{H}_{0}$, given in Definition \ref{boundaryhilberts}, pushed forward by the $K\times K$ intertwining isomorphism $e^{-i\tau\hat{h}}$.
\end{proof}

\bigskip
For $\tau\in\mathbb{C}^{+}\cup\mathbb{R}$, consider the orthonormal basis $\{\rho_{ab}^{\tau},\rho\in\hat{K},a,,b=1,\dots,\dim\pi\}$ for $\mathcal{H}_{\tau}$, where
\[
\rho_{ab}^{\tau}=a_{\rho}(\tau)^{-1}\rho(xe^{\tau u})_{ab}\beta_{\tau}(Y)\otimes\sqrt{\Omega_{\tau}}.
\]

To fix phase ambiguities, we will define the operators $E(\tau,h)$ by setting $\lambda_{\rho}=\frac{a_{\rho}(0)}{a_{\rho}(\tau)},\ \rho\in\hat{K}$. 

\begin{defn}
The generalized $h$-CST associated to the function $h$ is the unitary transform $U_{\tau}$ in (\ref{gencst}). 
\end{defn}

\begin{thm}
The generalized CST, $U_{\tau}$, is a unitary isomorphism between $\mathcal{H}_{0}$ and $\mathcal{H}_{_{\tau}}$, with
\[
U_{\tau}(\rho_{ab}^{0})=\rho_{ab}^{\tau},\ \rho\in\hat{K}.
\]
\end{thm}

\begin{proof}
The results follows directly from Propositions \ref{evolu}, \ref{blockdiag} and \ref{admissible}. 
\end{proof}

\begin{rem}
Note, that $a_{\rho}(\tau)$, and therefore $E(\tau,h)$, depends only on $\tau_{2}$. Moreover, the unitarity of $U_{\tau}$, and the property that it interwines the $K\times K$ actions, fixes the operator $E(\tau,h)$ uniquely up to (unitary) phase factors which may be chosen arbitrarily for each block of the decomposition $\mathcal{H}_{0}=\overline{\oplus_{\rho\in\hat{K}}V_{\rho,\rho}^{0}}$.
\end{rem}

\begin{rem}\label{remarkably}
Remarkably, as shown in \cite{Hall02,Florentino-Matias-Mourao-Nunes05,Florentino-Matias-Mourao-Nunes06} for $\tau=it,\ t>0$ and $h=\frac{1}{2}|Y|^{2}$ the quadratic energy function, one has
\[
E(it,h)=e^{-tQ(h)}
\]
where $Q(h)$ is the $t$-independent operator
\[
Q(h)=-\frac{1}{2}\Delta+\frac{|\rho|}{2}^{2},
\]
and $\Delta$ is the (negative-defined) Laplacian on $K$. In this case, $U_{it}$ is Hall's CST. Thus, for $h(Y)=\frac12 |Y|^2$, and $\tau=it$, we see that Hall's CST can be written in the form
\begin{equation}
C_{t}=U_{t}=e^{t\hat{h}}\circ e^{-t(-\frac{\Delta}{2}+\frac{|\rho|^{2}}{2})}.\label{cstdecomp}
\end{equation}
Note that the operator $Q(h)$ is given by the Schr\"odinger quantization of the Hamiltonian function $h$. Here, we can clearly identify the operator $e^{t\hat{h}}$ as responsible for the ``analytic continuation'' part of the CST, after application of the heat kernel semigroup. In this example, the fact that
\[
E(it,h)^{-1}\frac{d}{dt}E(it,h)=itQ(h)
\]
for a fixed operator $Q(h)$, is a consequence of the fact that the function $a_{\rho}(\tau)$ giving the norms of the polarized states $Tr(A\rho(xe^{\tau u}))\beta_{\tau}(Y)\otimes\sqrt{\Omega_{\tau}}$ is as an exponential of $\tau_{2}$ times a constant \cite{Hall02,Florentino-Matias-Mourao-Nunes06}, so that $E(\tau,h)\circ E(\tau^{\prime},h)=E(\tau+\tau^{\prime},h)$. In general, however, i.e. for nonquadratic complexifier $h$, this property will not be present, as we will show in the next section. In particular, one loses the semigroup property.
\end{rem}

\begin{rem}
We note that analogous ``generalized CST's'', associated to an $Ad_{K}$-invariant measure on $T^{*}K$ such that the analytic continuations of matrix elements of irreducible representations of $K$ are square integrable, have appeared in Hall's original paper \cite{Hall94}. Here, we are giving a geometric quantization incarnation to a large family of such transforms. 
\end{rem} 

The family of quantizations described above and the ``decomposition'' of Hall's CST in (\ref{cstdecomp}) suggests a path for addressing the issue raised in \cite{hall00} of finding the Stone-von Neumann explanation for the CST. Recall that the Stone-von Neumann theorem was generalized by Mackey \cite{mackey49}, as follows. A covariant pair $(R,\gamma)$ of representations of $K$ and $C^{0}(K)$ on a Hilbert space $\mathcal{H}$ is a unitary representation of $K$ on $\mathcal{H}$, together with a $*$-representation of $C^{0}(K)$, $\gamma$, such that $R(x)\gamma(f)R(x^{-1})=\gamma(x\cdot f)$, where $x\cdot f(x^{\prime})=f(x^{-1}x^{\prime}),\ x\in K,\ f\in C^{0}(K)$. Mackey's theorem states that any covariant pair of representations of $K$ and $C^0(K)$ on a Hilbert space is unitarily equivalent to a countable direct sum of the standard representations on $L^{2}(K)$ \cite{mackey49,rosenberg04}. In the standard representation, functions in $C^{0}(K)$ act on $L^{2}(K)$ by pointwise multiplication and $x\in K$ acts by left translating the $L^{2}$ function's argument by $x^{-1}$. Thus, the Hilbert space for the standard covariant pair is just $\mathcal{H}_{0}\cong L^{2}(K,dx)$ with the action of $K\times K$ and multiplication by functions.

For each $\tau\in\mathbb{C}^{+}\cup\mathbb{R}$, the action of the subgroup $K\times\{1_{K}\}\subset K\times K$ on the Hilbert spaces $\mathcal{H}_{\tau}$ gives part of a covariant pair. Moreover, these actions are intertwined by the unitary operators $U_{\tau}.$ Of course, the standard action of $C^{0}(K)$ on $\mathcal{H}_{0}$ can also be conjugated by the operators $U_{\tau}$ to define a covariant pair on $\mathcal{H}_{\tau}$. In this way, the problem of giving a Stone-von Neumann explanation for Hall's CST translates into interpreting the appearance of the Schr\"odinger operator $Q(h)$ in Remark \ref{remarkably} in terms of geometric quantization. For the generalized $h$-CSTs a similar interpretation is needed for the operators $E(\tau,h).$ Recall that in the case of flat space and of abelian varieties, the existence of unitary BKS pairing maps between Hilbert spaces for quantizations in different polarizations is explained by a Stone-von Neumann theorem. (For abelian varieties the relevant finite Heisenberg group is a finite analog of the standard covariant pair for $S^{1}$.) We will address these issues in \cite{Kirwin-Mourao-Nunes}.

\section{Nontransitivity for general complexifiers $h$}

In general, the operator $E(\tau,h)^{-1}\frac{d}{d\tau_{2}}E(\tau,h)$ is not $\tau_{2}$-independent, as is the case for the usual CST associated to the complexifier $\frac{1}{2}\kappa=\frac{1}{2}|Y|^{2}$. As mentioned above, this is a consequence of the fact that the functions $a_{\rho}(\tau)$ are in general not of the form $e^{c_{\rho}\tau_{2}}$, for some constant $c_{\rho}$. Let us show this already in a simple example with $K=S^{1}.$ In this case, even though $K$ is not simple, all formulas go through with $\eta(Y)=1$.

Let $h(Y)$ be a polynomial in $Y$ with positive definite second derivative bounded away from 0. Then, $\mathcal{H}_{\tau}$ has an orthogonal basis
\[
\{e_{n}=e^{n(i\theta+\tau u)}e^{i\tau(yu-h(y))}\}_{n\in\mathbb{Z}}.
\]
The corresponding norms are given by, with $y=iw,u=is$,
\[
a_{n}(\tau)^{2}=||e_{n}||^{2}=\int_{\mathbb{R}}e^{-2n\tau_{2}s}e^{-2\tau_{2}(sw-h)}\tau_{2}^{\frac{1}{2}}\sqrt{h^{^{\prime\prime}}(w)}dw.
\]

Let us take, for example, $h(w)=\frac{w^{2}}{2}+\frac{w^{4}}{4}$, so that $s=w+w^{3}$. Then,
\[
a_{n}(\tau)^{2}=\int_{\mathbb{R}}e^{-2n\tau_{2}(w+w^{3})}e^{-\tau_{2}(w^{2}+\frac{w^{4}}{2})}\tau_{2}^{\frac{1}{2}}\sqrt{3w^{2}}dw.
\]
It is easy to verify, for example using the Laplace approximation and the large $\tau_{2}$ asymptotics of $a_{n}(\tau)^{2}$, that there is no constant $c$ such that
\[
\frac{da_{n}(\tau)}{d\tau_{2}}=ca_{n}(\tau)^{2},
\]
so that the operators $E(\tau,h)$ will not have nice transitivity properties in this case.\bigskip{}

\noindent\textbf{\large Acknowledgments:} We thank Brian Hall for discussions. This work was supported in part by the European Science Foundation (ESF) grant ``Interactions of Low-Dimensional Topology and Geometry with Mathematical Physics (ITGP)''. The last two authors were supported by the project PTDC/MAT/119689/2010.

\providecommand{\bysame}{\leavevmode\hbox to3em{\hrulefill}\thinspace}

\end{document}